\documentclass[final,a4paper, 12pt]{article}

\usepackage{amsfonts}
\usepackage{amsmath}
\usepackage{amssymb}
\usepackage{amscd}
\usepackage{amsthm}

\usepackage{bm}
\usepackage{cite}
\usepackage{showkeys}
\usepackage[english]{babel}
\usepackage{dblaccnt}
\usepackage{accents}
\usepackage{graphicx}
\usepackage{array}

\usepackage{algorithm}
\usepackage{algorithmicx}
\usepackage{algpseudocode}

\usepackage[mathscr]{euscript}
\usepackage{color}
\usepackage{fancyhdr}
\usepackage{dsfont}
\usepackage[a4paper,scale=0.752]{geometry}

\newtheorem{definition}{Definition}
\newtheorem{lemma}[definition]{Lemma}
\newtheorem{theorem}[definition]{Theorem}

\newtheorem{assumption}[definition]{Assumption}

\newtheorem{remark}[definition]{Remark}

\newcommand*{\N}{\ensuremath{\mathbb{N}}}
\newcommand*{\Z}{\ensuremath{\mathbb{Z}}}
\newcommand*{\R}{\ensuremath{\mathbb{R}}}

\renewcommand{\i}{\mathrm{i}}
\renewcommand{\phi}{\varphi}
\renewcommand{\d}[1]{\,\mathrm{d}#1 \,}
\newcommand{\dS}{\,\mathrm{dS} \,}

\newcommand{\J}{\mathcal{J}} 

\renewcommand{\Re}{\mathrm{Re}\,}
\renewcommand{\Im}{\mathrm{Im}\,}

\newcommand{\W}{{W_{\hspace*{-1pt}\Lambda}}} 
\newcommand{\Wast}{{W_{\hspace*{-1pt}\Lambda^\ast}}} 

\renewcommand{\rho}{{\varrho}}
\renewcommand{\epsilon}{{\varepsilon}}
\newcommand{\loc}{{\mathrm{loc}}}

\newlength{\dhatheight}

\usepackage{color}

\begin{document}

\sloppy

\title{Reconstruction of Local Perturbations in Periodic Surfaces}
\author{Armin Lechleiter\thanks{Center for Industrial Mathematics, University of Bremen
; \texttt{lechleiter@math.uni-bremen.de}} \and 
Ruming Zhang\thanks{Center for Industrial Mathematics, University of Bremen
; \texttt{rzhang@uni-bremen.de}}\ \thanks{corresponding author}}
\date{}
\maketitle

\begin{abstract}
This paper concerns the inverse scattering problem to reconstruct a local perturbation in a periodic structure. Unlike the periodic problems, the periodicity for the scattered field no longer holds, thus classical methods, which reduce quasi-periodic fields in one periodic cell, are no longer available. Based on the Floquet-Bloch transform, a numerical method has been developed to solve the direct problem, that leads to a possibility to design an algorithm for the inverse problem. The numerical method introduced in this paper contains two steps. The first step is initialization, that is to locate the support of the perturbation by a simple method. This step reduces the inverse problem in an infinite domain into one periodic cell. The second step is to apply Newton-CG method to solve the associated optimization problem. The perturbation is then approximated by a finite spline basis. Numerical examples are given at the end of this paper, shows the efficiency of the numerical method.
\end{abstract}

\section{Introduction}

In this paper, we will introduce the numerical method of the inverse scattering problem in a locally perturbed periodic structure. Both the direct and inverse scattering problems in periodic structures have been studied in the last few years, especially for the case that the incident fields are quasi-periodic, e.g. plain waves. A classical way is to reduce the problems defined in an infinite periodic domain into one periodic cell, then the finitely defined problems could be solved in normal methods. However, if the periodicity of the solutions is destroyed, i.e., the incident field is not (quasi-)periodic, or the periodic structure is perturbed, the classical methods are no longer available and new techniques are needed.

The Floque-Bloch transform has been applied to perturbed periodic structures in \cite{Coatl2012}. The direct scattering problems with non-periodic incident fields (Herglotz wave functions) have been studied in \cite{Lechl2015e} and \cite{lechl2016} theoretically, for numerical method see \cite{Lechl2016a,Lechl2016b}; problems with locally perturbed periodic surfaces have also been studied, for theoretical part see \cite{lechl2016} and for numerical method see \cite{lechl2017}. The Bloch transform could also be applied to waveguide problems, see \cite{Hadda2016}. In this paper, the analysis and numerical solutions  of the direct problems are based on these results.

The numerical method developed in this paper is a combination of an initialization and an iteration scheme. The initialization step is to locate the perturbation from a relatively larger area. As was introduced by Ito, Jin and Zou in \cite{Ito2012}, a sampling method that only involves one incident field and a simple evaluation of an integration, could roughly locate the inhomogeneity embedded in  free space. For more results for this method, see \cite{Ito2012a, Ito2013, Li2013}. Followed by their idea, we will design an initialization algorithm to locate the perturbation, such that we could continue the iteration step in a finite domain. In the next step, we will approximate the perturbation by finite number of spline basis, and rewritten the inverse problem as an optimization problem. In this paper, we will apply Newton-CG method (see \cite{Engl1996}) to solve the associated optimization method.

The rest of the paper is organized as follows. In Section 2, a description of the direct problem is made, together with the Green's function in periodic structures. The inverse problem is formulated in Section 3, and the Fr\'{e}chet derivative and its adjoint operator are studied. In Section 4, we will introduce the numerical methods for the inverse problem. In Section 5, several numerical examples are shown to illustrate the efficiency of the numerical method.

\section{Direct Scattering Problem}\label{sec:dir}
\subsection{Formulation}
Suppose $\Gamma:=\{(x_1,\zeta(x_1)):\,x_1\in\R\}$ is a $\Lambda$-periodic interface in $\R^2$ defined by the function $\zeta$ that is $\Lambda$-periodic. $\Gamma_p:=\{(x_1,\zeta_p(x_1)):\,x_1\in\R\}$ is a local perturbation of $\Gamma$, where the function $\zeta-\zeta_p$ has a compact support in $\R$. 
 $\Gamma_H=\R\times\{H\}$ is a straight line above $\Gamma$ and $\Gamma_p$, where $H>\max\{\sup_{x_1\in\R}\zeta(x_1),\,\sup_{x_1\in\R}\zeta_p(x_1)\}$. For some $h>H$, define $\Gamma_h=\R\times\{h\}$. Define the unbounded domains $\Omega$, $\Omega_p$, $\Omega_H$, $\Omega_H^p$ as
\begin{eqnarray*}
&&\Omega=\{(x_1,x_2)\in\R^2:\,x_2>\zeta(x_1)\};\quad
\Omega_p=\{(x_1,x_2)\in\R^2:\,x_2>\zeta_p(x_1)\};\\
&&\Omega_H=\{(x_1,x_2)\in\Omega:\,x_2<H\};\quad
 \Omega^p_H=\{(x_1,x_2)\in\Omega_p:\,x_2<H\}
\end{eqnarray*} 
Define $\Lambda^*:=2\pi/\Lambda$, and the define periodic cell s
\begin{equation*}
\W=\left(-\frac{\Lambda}{2},\frac{\Lambda}{2}\right]\text{ and }\Wast=\left(-\frac{\Lambda^*}{2},\frac{\Lambda^*}{2}\right]=\left(-\frac{\pi}{\Lambda},\frac{\pi}{\Lambda}\right].
\end{equation*} 
Let $\Gamma^\Lambda$, $\Gamma^\Lambda_H$, $\Omega^\Lambda$, $\Omega^\Lambda_H$ denote the domains $\Gamma$, $\Gamma_H$, $\Omega$, $\Omega_H$ restricted in one period $\W\times\R$.

\begin{assumption}
For simplicity, we assume that the functions $\zeta$ and $\zeta_p$ are smooth enough, and the support of $\zeta-\zeta_p$ lies in one periodic cell, i.e., there is some  $J\in\Z$ such that $\rm{supp}(\zeta-\zeta_p)\subset \W+J\Lambda$. 
\end{assumption}

Let the incident field $u^i$ be a downward propagating Herglotz wave function, i.e., 
\begin{equation*}
u^i(x)=\int_{-\pi/2}^{\pi/2} e^{ik(x_1\cos\theta -x_2\sin\theta )}g(\theta)\,d\theta,
\end{equation*}
where $g\in L^2_{\cos}(-\pi/2,\pi/2)$, $L^2_{\cos}(-\pi/2,\pi/2)$ is the closure of the space $C_0^\infty(-\pi/2,\pi/2)$ in the norm
\begin{equation*}
\left\|\phi\right\|_{L^2_{\cos}(-\pi/2,\pi/2)}:=\left[\int_{-\pi/2}^{\pi/2}\left|\phi(\theta)\right|^2/\cos\theta\d\theta\right]^{1/2}.
\end{equation*}
Then the total field $u_p$ satisfies the equations
\begin{eqnarray}
&&\Delta u_p+k^2u_p=0\quad\text{ in }\Omega_p,\label{eq:orig1}\\
&& u_p=0\quad\text{ on }\Gamma_p,\label{eq:orig2}
\end{eqnarray}
with the scattered field $u^s_p:=u_p-u^i$ is the upward propogating field, satisfies
\begin{equation}
\frac{\partial u^s_p}{\partial x_2}=T u^s_p\quad\text{ on }\Gamma_H,
\end{equation}
where $T$ is defined by
\begin{equation*}
(Tv)(x_1)=\frac{\i}{\sqrt{\Lambda}}\int_\R\sqrt{k^2-|\xi|^2}e^{\i x_1\Lambda^*\xi}\hat{v}(\xi)d\xi,\quad \hat{v}(\xi)=\frac{1}{\sqrt{\Lambda}}\int_\R e^{-\i\Lambda^*\xi x_1}v(x_1,H)dx_1.
\end{equation*}
It is a bounded operator from $H_r^{1/2}(\Gamma_H)$ to $H_r^{-1/2}(\Gamma_H)$ for all $|r|<1$, see \cite{Chand2010}. 
Thus the total field $u_p$ satisfies the following boundar condition on $\Gamma_H$: 
\begin{equation}
\frac{\partial u_p}{\partial x_2}=T(u_p|_{\Gamma_H})+\left[\frac{\partial u^i}{\partial x_2}-T(u^i|_{\Gamma_H})\right]\quad\text{ on }\Gamma_H.\label{eq:orig3}
\end{equation}

Recall the definition of the function space $H_r^s(\R)$. The space is defined as the closure of $C_0^\infty(\R)$ with the norm
\begin{equation*}
\left\|\phi\right\|_{H_r^s(\R)}=\left\|\left[(1+|x|^2)^{r/2}\phi(x)\right]\right\|_{H^s(\R)}.
\end{equation*}
for any $s,r\in\R$.  Define the space $\widetilde{H}^1_r(\Omega_H)$ and $\widetilde{H}^1_r(\Omega_H^p)$ as
\begin{equation*}
\widetilde{H}^1_r(\Omega_H):=\left\{f\in H^1_r(\Omega_H):\, f\big|_\Gamma=0\right\};
\quad
\widetilde{H}^1_r(\Omega_H^p):=\left\{f\in H^1_r(\Omega_H^p):\, f\big|_{\Gamma_p}=0\right\},
\end{equation*}
the variational formulation of \eqref{eq:orig1}-\eqref{eq:orig3} is to find $u_p\in\widetilde{H}^1_r(\Omega_H^p)$ such that
\begin{equation}\label{eq:orig_var}
\int_{\Omega_H^p}\left[\nabla u_p\cdot\nabla\overline{v}-k^2 u_p\overline{v}\right]\d x-\int_{\Gamma_H}T^+(u_p|_{\Gamma_H})\overline{v}\d s=\int_{\Gamma_H}f\overline{v}\d s,
\end{equation}
for all $v\in\widetilde{H}^1_r(\Omega_H^p)$, where $f=\frac{\partial u^i}{\partial x_2}-T^+(u^i|_{\Gamma_H})\in H_r^{-1/2}(\Gamma_H)$. From \cite{Chand2010}, the problem is uniquley soluable in $\widetilde{H}^1_r(\Omega_H^p)$ for any $|r|<1$ and $u^i\in H^1_r(\Omega_H^p)$.

\begin{theorem}\label{th:solv}
For $|r|<1$ and any incident field $u^i\in H^1_r(\Omega_H^p)$, the variational problem \eqref{eq:orig_var} is uniquely solvable in $u_p\in\widetilde{H}^1_r(\Omega_H^p)$.
\end{theorem}

\begin{remark}
We can also define the total and scattered fields with the non-perturbed interface $\Gamma$, and the unique solvability in Theorem \ref{th:solv} also holds for this situation. The total field is denoted by $u_0$ and the scattered field is $u^s_0$.
\end{remark}

\subsection{Bloch Transform}

Define the Bloch transform in the periodic domain $\Omega$ for $\phi\in C_0^\infty(\Omega)$
\begin{equation*}
\J_\Omega \phi(\alpha,x)=\left[\frac{\Lambda}{2\pi}\right]^{1/2}\sum_{j\in\Z}\phi\left(x+\left(\begin{matrix}
\Lambda j\\0
\end{matrix}
\right)\right) e^{\i \Lambda j\alpha}.
\end{equation*}
Define the functions space $H_0^r(\Wast;H_\alpha^s(\Omega^\Lambda_H))$ by the closure of $C_0^\infty(\Wast\times\Omega^\Lambda_H)$ with the following norm when $r\in\N$
\begin{equation*}
\left\|\psi\right\|_{H_0^r(\Wast;H_\alpha^s(\Omega^\Lambda_H))}=\left[\sum_{\gamma=0}^r\int_\Wast\left\|\partial^\gamma_\alpha\psi(\alpha,\cdot)\right\|^2_{H_\alpha^s(\Omega^\Lambda_H)}\d\alpha\right]^{1/2}.
\end{equation*}
The definition could be extended to all $r\in\R$ by interpolation and duality arguments. With this function space,  the operator has the following properties, see \cite{lechl2016}.
\begin{theorem}
The Bloch transform $\J_\Omega$ extends to  an isomorphism between $H^s_r(\Omega_H)$ and $H^r_0(\Wast;H^s_\alpha(\Omega^\Lambda_H))$. Further, $\J_\Omega$ is an isometry for $s=r=0$ with inverse
\begin{equation*}
(\J_\Omega^{-1}w)\left(x+\left(\begin{matrix}
\Lambda j\\0
\end{matrix}
\right)\right)=\left[\frac{\Lambda}{2\pi}\right]^{1/2}\int_\Wast w(\alpha,x)e^{\i\alpha\Lambda j}\d \alpha,\quad x\in\Omega^\Lambda_H
\end{equation*}
and the inverse transform equals to the adjoint operator of $\J_\Omega$.
\end{theorem}

Bloch transform only defined on periodic domains, so following \cite{lechl2017}, define the diffeomorphism mapping $\Omega$ onto $\Omega_p$ as
\begin{equation*}
\Phi_p:\, x\mapsto\left(x_1,x_2+\frac{(x_2-H)^3}{(\zeta(x_1)-H)^3}\left(\zeta_p(x_1)-\zeta(x_1)\right)\right),
\end{equation*}
then  $\rm{supp}(\Phi_p-I_2)\subset \Omega^\Lambda_H+J(\Lambda,0)^T$. Define the transformed total field $u_T=u_p\circ\Phi_p\in \widetilde{H}^1_r(\Omega_H)$, then from \eqref{eq:orig_var} it satisfies the following variational problem
\begin{equation}
\label{eq:var}
\int_{\Omega_H^p}\left[A_p\nabla u_T\cdot\nabla\overline{v_T}-k^2 c_p u\overline{v_T}\right]\d x-\int_{\Gamma_H}T^+(u_T\big|_{\Gamma_H})\overline{v_T}\d s=\int_{\Gamma_H}f\overline{v_T}\d s
\end{equation}
for all $v_T\in \widetilde{H}^1(\Omega_H)$, where
\begin{eqnarray*}
&& A_p(x):=\big|\det\nabla\Phi_p(x)\big|\left[\left(\nabla\Phi_p(x)\right)^{-1}\left(\big(\nabla\Phi_p(x)\right)^{-1}\big)^T\right]\in L^\infty(\Omega_H,\R^{2\times 2})
\\
&& c_p(x):=\big|\det\nabla\Phi_p(x)\big|\in L^\infty(\Omega_H).
\end{eqnarray*}
It is easy to deduce that the support of $A_p-I_2$ and $c_p-1$ are all subsets of $\Omega^\Lambda_H+J(\Lambda,0)^T$. From the solution $u_p\in \widetilde{H}^1_r(\Omega_H^p)$, the transformed function $u_T\in\widetilde{H}^1_r(\Omega_H)$.

Let $w_T=\J_\Omega u_T\in L^2(\Wast;\,\widetilde{H}^1_\alpha(\Omega^\Lambda_H))$ and $v_B\in L^2(\Wast;\,\widetilde{H}^1_\alpha(\Omega^\Lambda_H))$, then it satisfies
\begin{equation}\label{eq:var_blo}
\int_\Wast a_\alpha(w_T(\alpha,\cdot),v_T(\alpha,\cdot))\d\alpha+\left[\frac{\Lambda}{2\pi}\right]^{1/2}b(\J^{-1}_\Omega w_T,\J^{-1}_\Omega v_T)=\int_\Wast\int_{\Gamma^\Lambda_H}f(\alpha,\cdot)\overline{v_T}\,\d s \d\alpha,
\end{equation}
where 
\begin{eqnarray*}
&&a_\alpha(u,v)=\int_{\Omega^\Lambda_H}\left[\nabla u\cdot\nabla\overline{v}-k^2 u\overline{v}\right]\,dx-\int_{\Gamma^\Lambda_H}(T_\alpha u)\overline{v}\,\d s,\\
&&b(u,v)=\int_{\Omega^\Lambda_H+J(\Lambda,0)^T}\left[(A_p-I)\nabla u\cdot \nabla\overline{v}-k^2(c_p-1)u\overline{v}\right]\,\d x,\\
&&f(\alpha,\cdot)=\frac{\partial (\J_\Omega u^i)(\alpha,\cdot)}{\partial x_2}-T_\alpha(\J_\Omega u^i)(\alpha,\cdot),\\
&&T_\alpha(\phi)=i\sum_{j\in\Z}\sqrt{k^2-|\Lambda^*j-\alpha|^2}\hat{\phi}(j)e^{\i(\Lambda^*j-\alpha)x_1},\,\phi=\sum_{j\in\Z}\hat{\phi}(j)e^{\i(\Lambda^*j-\alpha)x_1}.
\end{eqnarray*}
For any $x_2>H$, for each fixed $\alpha$, the $\alpha$-quasi-periodic solution $w_T(\cdot,\alpha)$ has the following representation, i.e., the Rayleigh expansion
\begin{equation}\label{eq:ray}
w_T(x,\alpha)=\sum_{j\in\Z}\hat{w}_T(j,\alpha)e^{\i(\Lambda^*j-\alpha)x_1+\i\beta_j x_2},\quad\beta_j=\sqrt{k^2-|\Lambda^* j-\alpha|^2}.
\end{equation}

From \cite{lechl2017}, the variational problem is uniquely solvable in certain conditions.

\begin{theorem}Suppose $\Gamma_p$ is the graph of a Liptshitz continuous function. 
\begin{enumerate}
\item The variational problem \eqref{eq:var_blo} is uniquely solvable in $H^r_0(\Wast;\,\widetilde{H}^1_\alpha(\Omega^\Lambda_H))$.
\item If $u^i\in H^2_r(\Omega_H^p)$ for $r\in[0,1)$ and $\zeta,\zeta_p\in C^{2,1}(\R,\R)$, then $w_B\in L^2(\Wast;\,H^2_\alpha(\Omega^\Lambda_H))$.
\end{enumerate}
\label{th:uni}
\end{theorem} 

\begin{remark}
For the fields with non-perturbed surface $\Gamma$, there is no need to do the transformation $\Phi_p$. We denote the Bloch transform of $u$ by $w=\J_\Omega u$, and Theorem \ref{th:uni} also holds for this problem.
\end{remark}

\subsection{Green's Functions in the Periodic Domain}

Suppose $G(x,y)$ is the Green's function located in $y\in \Omega_H$, then it satisfies
\begin{eqnarray}
\Delta G(x,y)+k^2G(x,y)&=&0,\quad x\in\Omega_H\setminus\{y\},\label{eq:green1}\\
G(x,y)&=&0,\quad x\in\Gamma,\label{eq:green2}\\
\frac{\partial G(x,y)}{\partial x_2}&=&T G(x,y),\quad x\in\Gamma_H,\label{eq:green3}
\end{eqnarray}
For $\Gamma=\R\times\{0\}$ we note that $G_0(x,y)=\Phi(x,y)-\Phi(x,y')$ is the incident half-space Green's function, where $\Phi(x,y)=\frac{\i}{4}H^{(1)}_0(k|x-y|)$, $y=(y_1,y_2)^T$ and $y'=(y_1,-y_2)^T$. From Theorem \ref{th:uni}, with the fact that $G_0(x,y)\in H_r^1(\R\times[0,H])$, the Green's function $G(x,y)$ is well-defined in $\Omega_H\setminus\{y\}$. Moreover, $G^s(x,y):=G(x,y)-G_0(x,y)\in H^1_r(\Omega_H)$.

\begin{theorem}[Theorem III.1, \cite{Lechl2008a}]
The Green's function is symmetric, i.e., $G(x,y)=G(y,x)$ with $x,\,y\in \Omega$.
\end{theorem}

Suppose $x_s\in\Omega_H$ and $x\in\Gamma_h$, then from the boundary condition \eqref{eq:green3}, if $G(x,x_s)$ has the following form of Fourier transform, i.e., 
\begin{equation}
G(x,x_s)=\int_\R e^{\i x_1\Lambda^*\xi}{\phi}_s(\xi,x_2)\d \xi,
\end{equation}
then its normal derivative on $\Gamma_h$ 
\begin{equation}
\frac{\partial G(x,x_s)}{\partial x_2}=\i\int_\R\sqrt{k^2-|\xi|^2} e^{\i x_1\Lambda^*\xi}{\phi}_s(\xi,x_2)\d \xi.
\end{equation}

\begin{theorem} Suppose $x_p,x_q$ are two different points in $\Omega_H$. 
Then the Green's function satisfies the following property
\begin{equation}
\Im\left[ G(x_p,x_q)\right]=-\Lambda\int_{-k}^k \sqrt{k^2-|\xi|^2}{\phi}_p(\xi,h)\overline{{\phi}}_q(\xi,h)\d \xi.
\end{equation}
\end{theorem}

\begin{proof} As $x_p\neq x_q$, for any    $0<\epsilon<|x_p-x_q|/3$, define $\Omega^\epsilon_p=\{x\in\Omega_H:\,|x-x_p|<\epsilon\}$ (and similar for $\Omega^\epsilon_q$), then $\Omega^\epsilon_p\cap \Omega^\epsilon_q=\emptyset$. Denote $\widetilde{G}^s(x,y)=G(x,y)-\Phi(x,y)$, then 
\begin{equation*}
\Delta \widetilde{G}^s(x,x_p)+k^2\widetilde{G}^s(x,x_p)=0,\quad\Delta G(x,x_q)+k^2G(x,x_q)=0\quad \text{ in }\Omega^\epsilon_p.
\end{equation*}
With these properties,
\begin{align*}
&\int_{\partial \Omega^\epsilon_p} \left(G(x,x_p)\frac{\partial \overline{G}(x,x_q)}{\partial \nu(x)}-\frac{\partial G(x,x_p)}{\partial\nu(x)}\overline{G}(x,x_q)\right)\d x\\
=&\int_{\partial \Omega^\epsilon_p} \left(\Phi(x,x_p)\frac{\partial \overline{G}(x,x_q)}{\partial \nu(x)}-\frac{\partial \Phi(x,x_p)}{\partial\nu(x)}\overline{G}(x,x_q)\right)\d x\\
=&\frac{\i}{4}\int_0^{2\pi}\left[H^{(1)}_0(k\epsilon)\frac{\partial \overline{G}(x,x_q)}{\partial \nu(x)}+k H^{(1)}_1(k\epsilon)\overline{G}(x,x_q)\right]\epsilon\d\theta.
\end{align*}
From the asympototic behaviers of $H^{(1)}_1(k\epsilon)$ and $H^{(1)}_1(k\epsilon)$, 
\begin{equation*}
\int_{\partial \Omega^\epsilon_p} \left(G(x,x_p)\frac{\partial \overline{G}(x,x_q)}{\partial \nu(x)}-\frac{\partial G(x,x_p)}{\partial\nu(x)}\overline{G}(x,x_q)\right)\d x\rightarrow \overline{G}(x_p,x_q),\text{ as }\epsilon\rightarrow 0.
\end{equation*}
Similarly, 
\begin{equation*}
\int_{\partial \Omega^\epsilon_p} \left(G(x,x_p)\frac{\partial \overline{G}(x,x_q)}{\partial \nu(x)}-\frac{\partial G(x,x_p)}{\partial\nu(x)}\overline{G}(x,x_q)\right)\d x\rightarrow -G(x_p,x_q),\text{ as }\epsilon\rightarrow 0.
\end{equation*}

From the Green's identity and Dirichlet boundary conditions, for small enough $\epsilon>0$
\begin{align*}
&\int_{\Gamma_H}\left(G(x,x_p)\frac{\partial \overline{G}(x,x_q)}{\partial x_2}-\frac{\partial G(x,x_p)}{\partial x_2}\overline{G}(x,x_q)\right)\d x\\
=&\left(\int_{\partial \Omega^\epsilon_p}+\int_{\partial \Omega^\epsilon_q}\right)\left(G(x,x_p)\frac{\partial \overline{G}(x,x_q)}{\partial \nu}-\frac{\partial G(x,x_p)}{\partial \nu}\overline{G}(x,x_q)\right)\d x,
\end{align*}
where the right hand side tends to $\overline{G}(x_p,x_q)-G(x_p,x_q)$ as $\epsilon\rightarrow 0$,
i.e.,
\begin{equation}
2\i\Im \left[ G(x_p,x_q)\right]=\int_{\Gamma_H}\left(G(x,x_p)\frac{\partial \overline{G}(x,x_q)}{\partial x_2}-\frac{\partial G(x,x_p)}{\partial x_2}\overline{G}(x,x_q)\right)\d x.
\end{equation}
 As the Green's function has the integral representation, 
\begin{equation*}
\begin{aligned}
&\int_{\Gamma_H}G(x,x_p)\frac{\partial \overline{G}(x,x_q)}{\partial x_2}\d x\\
= &-\i\int_\R\left[\int_\R e^{\i x_1\Lambda^*\xi}{\phi}_p(\xi,h)\d \xi\right]\left[\int_\R\overline{\sqrt{k^2-|\eta|^2}}e^{-\i x_1\Lambda^*\eta}\overline{{\phi}}_q(\eta,h)\d \eta\right] \d x_1\\
= &-\i\int_\R \left(\int_\R \left[\int_\R e^{\i x_1\Lambda^*(\xi-\eta)}\d x_1\right]\overline{\sqrt{k^2-|\eta|^2}}\overline{{\phi}}_q(\eta,h)\d \eta\right){\phi}_p(\xi,h)\d \xi\\
=&-\Lambda\i\int_\R \int_\R \delta(\xi-\eta)\overline{\sqrt{k^2-|\eta|^2}}\overline{{\phi}}_q(\eta,h)\d \eta{\phi}_p(\xi,h)\d \xi\\
=&-\Lambda\i\int_\R  \overline{ \sqrt{k^2-|\xi|^2}}\overline{{\phi}}_q(\xi,h){\phi}_p(\xi,h)\d \xi,
\end{aligned}
\end{equation*}
thus  
\begin{align*}
\Im\left[ G(x_p,x_q)\right]&=-\Lambda\int_\R \Re\left[\sqrt{k^2-|\xi|^2}\right]{\phi}_p(\xi,h)\overline{{\phi}}_q(\xi,h)\d \xi\\
&=-\Lambda\int_{-k}^k \sqrt{k^2-|\xi|^2}{\phi}_p(\xi,h)\overline{{\phi}}_q(\xi,h)\d \xi.
\end{align*}

\end{proof}

Define the function
\begin{equation}
\mathcal{I}(x_p,x_q):=\int_{\Gamma_h}G(x,x_p)\overline{G}(x,x_q)\d x,
\end{equation}
the following property is also easy to prove
\begin{equation}
\mathcal{I}(x_p,x_q)=\Lambda\int_R {\phi}_p(\xi,h)\overline{{\phi}}_q(\xi,h)\d \xi.
\end{equation}
When $h$ is large enough, from the definition of UPRC,
\begin{equation*}
\mathcal{I}(x_p,x_q)\rightarrow \Lambda\int_{-k}^k {\phi}_p(\xi,h)\overline{{\phi}}_q(\xi,h)\d \xi.
\end{equation*}
When $x_q$ and $x_p$ are close enough, $\phi_p(\xi,h)\rightarrow\phi_q(\xi,h)$, thus
\begin{eqnarray*}
&&\Im\left[ G(x_p,x_q)\right]\rightarrow-\Lambda\int_{-k}^k \sqrt{k^2-|\xi|^2}|{\phi}_p(\xi)|^2\d \xi\\
&&\mathcal{I}(x_p,x_q)\rightarrow \Lambda\int_{-k}^k |{\phi}_p(\xi)|^2\d \xi\geq k \left|\Im\left[ G(x_p,x_q)\right]\right|.
\end{eqnarray*}

As $G_0(x_p,x_q)$ has a singularity at $x_p=x_q$, the imaginary part of the Green's function $\left|\Im\left[G(x_p,x_q)\right]\right|$ may possess a sigularity, or a local maximum at that point. From the approximate behavior of $\mathcal{I}(x_p,x_q)$ when  when $x_p$ falls in a small enough neighbourhood of $x_q$, it could be expected to obtain relatively large values.

\section{The Inverse Problem}

Suppose $\zeta\in C^{2,1}(\R,\R)$ and is a $\Lambda$-periodic knowns function, and define the space 
\begin{equation*}
X=\Big\{p\in C^{2,1}(\R,\R):\,\rm{supp}(p)\subset \W+J\Lambda\Big\}.
\end{equation*}
In the following, we will always assume that $p\in X$ and $\zeta_p:=\zeta+p$, and $u^i\in H^2(\Omega_h^p)$, thus from Theorem \ref{th:uni}, $w_B\in L^2(\Wast;\,H^2_\alpha(\Omega^\Lambda_h))$. Then the functions $u_T=\J^{-1}w_B\in H^2(\Omega_h)$, $u_p\in H^2(\Omega^p_h)$. Define the following scattering operator
\begin{equation*}
\begin{aligned}
S:\quad X & \rightarrow  L^2(\Gamma_h)\\
p & \mapsto u_p^s|_{\Gamma_h}
\end{aligned}
\end{equation*}

Suppose data is measured on $\Gamma_H$, which is the scattered field with some noise, denoted by $U$, then the inverse problem is described as follows.

\noindent
\textbf{Inverse Problem: }Given a measured data $U$, to find $p\in X$ such that 
\begin{equation}\label{eq:ip}
\|Sp-U\|^2_{L^2(\Gamma_h)}=\min_{p^*\in X}\|Sp^*-U\|^2_{L^2(\Gamma_h)}.
\end{equation}

In the rest part of this section, we will describe some properties of the scattering operator.

\begin{theorem}
The operator $S$ is differentiable, and its derivative $DS$ has the following representation, i.e.,
\begin{equation}
\begin{aligned}
DS:\quad X & \rightarrow & L^2(\Gamma_h)\\
h &\mapsto & u'|_{\Gamma_h}
\end{aligned}
\end{equation}
where $u'\in H^1(\Omega_h^p)$ satisfies
\begin{eqnarray}
\label{eq:der1}&&\Delta u'+k^2u'=0\quad\text{ in }\Omega_p,\\
\label{eq:der2}&&u'=-\frac{\partial u_p}{\partial \nu}h\nu_2\quad\text{ on }\Gamma_p,\\
\label{eq:der3}&&\frac{\partial u'}{\partial x_2}=Tu'\quad\text{ on }\Gamma_h.
\end{eqnarray}
$u_p$ is the total field of the scattering problem \eqref{eq:orig1}-\eqref{eq:orig3}. 
\end{theorem}

The proof is similar to that in \cite{Kirsc1993b} so we omit it here.


Now we will study the property of $DS$, before that, the following property of the DtN operator is needed. 

\begin{lemma}
$(\cdot,\cdot)$ is the duality pairing defined by the inner product of $L^2(\Gamma_h)$. The DtN operator $T$ satisfies $\left(Tu,\overline{v}\right)=\left(u,\overline{Tv}\right)$.
\end{lemma}
 
\begin{proof}
Suppose $u,v\in C_c^\infty(\Gamma_h)$, then from the definition of the DtN map,
\begin{eqnarray*}
(Tu)(x_1)=\frac{\i}{\sqrt{\Lambda}}\int_\R \sqrt{k^2-|\xi|^2}e^{\i x_1\Lambda^*\xi}\hat{u}(\xi)\,\d\xi,\quad \hat{u}(\xi)=\frac{1}{\sqrt{\Lambda}}\int_\R e^{-\i\Lambda^*\xi x_1} u(x_1)\, dx_1;\\
(Tv)(x_1)=\frac{\i}{\sqrt{\Lambda}}\int_\R \sqrt{k^2-|\xi|^2}e^{\i x_1\Lambda^*\xi}\hat{v}(\xi)\,\d\xi,\quad \hat{v}(\xi)=\frac{1}{\sqrt{\Lambda}}\int_\R e^{-\i\Lambda^*\xi x_1} v(x_1)\, dx_1.
\end{eqnarray*}
Then we have the following calculations.
\begin{eqnarray*}
\left(Tu,\overline{v}\right)&=&\int_\R (Tu)(x_1)v(x_1)\d x_1\\
&=&\frac{\i}{\sqrt{\Lambda}}\int_\R\int_\R\sqrt{k^2-|\xi|^2}e^{\i x_1\Lambda^*\xi}\hat{u}(\xi)v(x_1)\,d\xi\,\d x_1\\
&=&\frac{\i}{\sqrt{\Lambda}}\int_\R\sqrt{k^2-|\xi|^2}\hat{u}(\xi)\left[\int_\R e^{\i x_1\Lambda^*\xi}v(x_1)\,dx_1\right]\,\d\xi\\
&=& i\int_\R\sqrt{k^2-|\xi|^2}\hat{u}(\xi)\hat{v}(-\xi)\,\d \xi.
\end{eqnarray*}
From similar procedure, 
\begin{eqnarray*}
\left(u,\overline{Tv}\right)=i\int_\R\sqrt{k^2-|\xi|^2}\hat{v}(\xi)\hat{u}(-\xi)\,d\xi=\left(Tu,\overline{v}\right),
\end{eqnarray*}
the proof is finished from the denseness of $C_c^\infty(\Gamma_h)$ in the space $L^2(\Gamma_h)$.
\end{proof}

With these results, the adjoint operator of $DS$ from $L^2(\Gamma_h)$ to $X^*$ is given in the following theorem.

\begin{theorem}
The adjoint operator of $DS$, denote by $M$, is then given by
\begin{equation}\label{eq:adjoint}
M\phi=-\Re\left[\frac{\partial \overline{u_p}}{\partial \nu}\frac{\partial \overline{z}}{\partial \nu}\right]\nu_2,
\end{equation}
where $\nu$ is the normal derivative upwards, $u_p$ is the total field of \eqref{eq:orig1}-\eqref{eq:orig3}, $z$ satisfies
\begin{eqnarray}
\Delta z+k^2 z=0\quad\text{ in }\Omega_p,\label{eq:adj1}\\
z=0\quad\text{ on }\Gamma_p,\label{eq:adj2}\\
\frac{\partial z}{\partial x_2}-Tz=\overline{\phi}\quad\text{ on }\Gamma_h.\label{eq:adj3}
\end{eqnarray}
\end{theorem}

\begin{proof}
For any $h\in X$ and $\phi\in L^2(\Gamma_H)$, as $X$ is the space of real-valued functions,
\begin{eqnarray*}
(h,M\phi)_{X}&=&\Re\big((DS)h,\phi\big)_{\Gamma_h}=\Re(u',\phi)_{\Gamma_h}=\Re\left(u',\frac{\partial \overline{v}}{\partial x_2}-\overline{Tv}\right)_{\Gamma_h}\\
&=&\Re\left[\left(u',\frac{\partial \overline{z}}{\partial x_2}\right)_{\Gamma_h}-\left(Tu',\overline{z}\right)_{\Gamma_h}\right]=\Re\left[\left(u',\frac{\partial \overline{z}}{\partial x_2}\right)_{\Gamma_h}-\left(\frac{\partial u'}{\partial x_2},\overline{z}\right)_{\Gamma_h}\right]\\
&=&\Re\int_{\Gamma_h}\left[u'\frac{\partial z}{\partial x_2}-\frac{\partial u'}{\partial x_2}z\right]\,\d s=\Re\int_{\Gamma_p}\left[u'\frac{\partial z}{\partial \nu}-\frac{\partial u'}{\partial \nu}z\right]\,\d s.
\end{eqnarray*}

Use the Helmholtz equation with the homogeneous Dirichlet boundary condition, i.e., $z=0$ on $\Gamma_p$, and the condition \eqref{eq:der2},
\begin{equation*}
(h,M\phi)_{X}=\Re\int_{\Gamma_p}\left[u'\frac{\partial z}{\partial \nu}\right]\,\d s=-\Re\int_{\Gamma_p}\left[h\nu_2\frac{\partial u_p}{\partial \nu}\frac{\partial z}{\partial \nu}\right]\,ds,
\end{equation*}
then we have the final results
\begin{equation*}
M\phi=-\Re\left[\frac{\partial \overline{u_p}}{\partial \nu}\frac{\partial \overline{z}}{\partial \nu}\right]\nu_2.
\end{equation*}

\end{proof}

\begin{remark}

To solve the problem \eqref{eq:adj1}-\eqref{eq:adj3}, we will study the variational forms of the Bloch transform of $z_T:=z\circ\Phi_p$. Followed by the skills in Section \ref{sec:dir}, $z_B:=\J_\Omega z_T$ satisfies the following variaitonal form 
\begin{equation}\label{eq:adj_var}
{\int_\Wast a_\alpha(w_B(\alpha,\cdot),v_B(\alpha,\cdot))d\alpha+\left[\frac{\Lambda}{2\pi}\right]^{1/2}b(\J^{-1}_\Omega w_B,\J^{-1}_\Omega v_B)=\int_\Wast\int_{\Gamma^\Lambda_H}g(\alpha,\cdot)\overline{v_B}\,dsd\alpha,}
\end{equation}
for all $v_B\in L^2(\Wast;\,\widetilde{H}^1_\alpha(\Omega_H))$, where 
\begin{eqnarray*}
g(\alpha,\cdot)=\left(\J_\Omega \overline{\phi}\right)(\alpha,\cdot).
\end{eqnarray*}

From the definition of Bloch tranform, the function 
\begin{eqnarray*}
g(\alpha,\cdot)=\J_\Omega\overline{\phi}(\alpha,x)&=&\left[\frac{\Lambda}{2\pi}\right]^{1/2}\sum_{j\in\Z}\overline{\phi}(x_1+\Lambda j,x_2)e^{-\i\Lambda j\alpha}\\
&=&\left[\frac{\Lambda}{2\pi}\right]^{1/2}\sum_{j\in\Z}\overline{\phi(x_1+\Lambda j,x_2)e^{\i\Lambda j\alpha}}\\
&=&\left[\frac{\Lambda}{2\pi}\right]^{1/2}\sum_{j\in\Z}\overline{\phi(x_1+\Lambda j,x_2)e^{-\i\Lambda j(-\alpha)}}\\
&=&\overline{\J_\Omega\phi(-\alpha,x)}.
\end{eqnarray*}
Thus the problem could be solved by solving the variational problem \eqref{eq:adj_var}.
\end{remark}

\section{Numerical Method for Inverse Problems}

The numerical method to inverse problems is devided into two parts. The first part is to  initialize  the location of the  perturbation, i.e., to find out the integer $J$ such that $\rm{supp}(\zeta-\zeta_p)\subset\W+J\Lambda$, with the idea from \cite{Ito2012}. With the known periodic cell where the perturbation located, the second step is to utilize the Newton-CG method to reconstruct the perturbed function. The two steps are introduced in the following two subsections.

\subsection{Initial Guess from the Sampling Method}

In this subsection, assume that $p\geq 0$. 
From Section \ref{sec:dir}, the transformed total field $u_T=u_p \circ \Phi_p$ in $H^1_0(\Omega_H)$ satisfies
\begin{align}\label{eq:varFormPhi}
  \int_{\Omega_H} & \left[ A_p \nabla u_T \cdot \nabla \overline{v} - k^2 \, c_p \, u_T \overline{v}\right] \d{x}   - \int_{\Gamma_H} \overline{v} \, T (u_T) \dS
  = \int_{\Gamma_H} f\overline{v} \dS \quad \text{for all $v \in \widetilde{H}^1(\Omega_H^p)$.} \nonumber
\end{align}


Denote the total field of the non-perturbed surface $\Gamma$ by $u_0$. Next we denote the difference between the transformed total field $u_T$ and  $u_0$ by $w \in H^1(\Omega_H)$. 
This difference satisfies 
\begin{align*}
  & \int_{\Omega_H} \left[ \nabla w \cdot \nabla \overline{v} - k^2 \, w \overline{v}\right] \d{x} 
  - \int_{\Gamma_H} \overline{v} \, T (w) \dS \\
  & \qquad = \int_{\Omega^H} \left[ (A_p-I) \nabla u_T \cdot \nabla \overline{v} - k^2 \, (c_p-1) u_T \overline{v}\right] \d{x}  \quad \text{for all $v \in \widetilde{H}^1(\Omega_H^p)$,}
\end{align*}
subject to vanishing boundary values $w=0$ in $H^{1/2}(\Gamma)$. 
Indeed, $u_T$ vanishes on $\Gamma$ since $u_p$ vanishes on $\Gamma_p$ and $u_0$ vanishes on $\Gamma$ as well. 
It hence follows from the representation formula for solutions to the Helmholtz equation that $w \in H^1(\Omega_H^p)$ can be represented as
\begin{align*}
  w(x) = \int_{\Omega_H^\Lambda} \left[ \big((I-A_p) \nabla_y G(x,y)\big) \cdot \nabla u_T(x) - k^2 \, (1-c_p) \, G(x,y) u_T(x) \right] \d{y} \qquad \text{in } \Omega_H. 
\end{align*}
Obviously, the latter representation formula extends to all points $x \in \Omega$, such that we extend $w$ by this formula to a function in $H^1_\loc(\Omega)$. 
Since this extension solves the Helmholtz equation and satisfies the angular spectrum representation encoded in $T$, it is well-known that the extension belongs to $H^1_0(\Omega_{h})$ for all $h \geq H$.

Taking the inner product on some measurement line $\Gamma_H$ of $w(x)$ with $\overline{G(x,x_p)}$, involving a parameter $x_p \in \Omega^H$ then shows by the results from Section 2 that 
\begin{align*}
  \int_{\Gamma_h} w(x)  \overline{G(x,x_p)} \d{S(x)}
  =  \int_{\Gamma_h}\int_{\Omega_H^\Lambda} & \left[ \big((I-A_p) \nabla_y G(x,y)\big) \cdot \nabla u_T(x) \right. \\
  & \qquad \left. - k^2 \, (1-c_p) \, G(x,y) u_T(x) \right] \overline{G(x,x_p)} \d{y} \d{x} .
\end{align*}
From the function $\mathcal{I}(y,x_p) = \int_{\Gamma_h} G(x,y) \overline{G(x,x_p)} \d{S(x)}$ and note that 
\begin{align*}
  \int_{\Gamma_h} w(x) \overline{G(x,x_p)} \d{S(x)} 
  =  \int_{\Omega_H^\Lambda} & \left[ \big((I-A_p(y)) \nabla u_T(y)\big) \cdot \nabla_y \mathcal{I}(y,x_p) \right. \\
  & \qquad\qquad - \left. k^2 \, (1-c_p(y)) \, u_T(y) \mathcal{I}(y,x_p) \right] \d{y}.
\end{align*}
Choosing any (small) neighborhood $\mathcal{N}$ of $(\Omega^H \setminus \Omega^H_p)$, there exists a diffeomorphism $\Phi_p$ supported in $\mathcal{N}$, such that the volumetric domain of integration in the last formula can be concentrated to $\mathcal{N}$. 
As $\Re\mathcal{I}(y,x_p)$ attains a maximum at $y=x_p$ this shows that the scalar product 
\begin{align*}
  \int_{\Gamma_h} (u_p-u) & \overline{G(\cdot,x_p)} \d{S(x)} 
  = \int_{\Gamma_h} (u_T-u) \overline{G(\cdot,x_p)} \d{S(x)} 
  = \int_{\Gamma_h} u_T \overline{G(\cdot,x_p)} \d{S(x)} \\
  &= \int_{\mathcal{N}} \left[ \big((I-A_p(y)) \nabla u_T(y)\big) \cdot \nabla_y \mathcal{I}(y,x_p)-  k^2 \, (1-c_p(y)) \, u_T(y) \mathcal{I}(y,x_p) \right] \d{y}
\end{align*}
can be expected to reach local maxima in $\mathcal{N}$, too.  

When applying this method, firstly, a relatively large domain $D$ is known to contain the perturbation. Assume that there is a positive integer $J_{max}\geq |J|$ such that $D\supset \cup_{j=-J_{max}}^{J_{max}}\left[\Omega^\Lambda_H+j(\Lambda,0)^T\right]$. Then the algorithm is to find the $J$ such that $\rm{supp}(p)\subset\Omega^\Lambda_H+J(\Lambda,0)^T$. Then the algorithm of this method could be concluded as follows.

\begin{algorithm}\caption{Sampling Method}
Input: Data $U$; domain $D$, $\mathcal{M}$ is a regular mesh for $D$.
\begin{enumerate}
\item Solve the non-perturbed problem, and get the scattered data $U_0:=u_0|_{\Gamma_h}$, define $w=U-U_0$.
\item For each $x_q\in \mathcal{M}$, evaluate $\mathcal{G}(x_q):=\left|\int_{\Gamma_h}w\overline{G(\cdot,x_q)}\d S\right|$.
\item Pick up the integer $J\in\left\{-\mathcal{J},\dots,\mathcal{J}\right\}$ such that $\mathcal{G}$ obtains the largest value in $\Omega^\Lambda_H+J(\Lambda,0)^T$.
\end{enumerate} 
\end{algorithm}

\subsection{Newton-CG Method}

In this section, we will discuss the Newton-CG method to solve \eqref{eq:ip}, i.e., to find some $p\in X$ such that it is a minimizer of $\|Sp-U\|^2_{L^2(\Gamma_h)}$. 

From the initialization step, the integer $J$ is known.  
Suppose $\{\phi_n\}_{n=1}^{+\infty}$ is a basis in the space $X$, and we are looking for the approximation of the function $p$ that belongs to the subspace of $X$, i.e.,
\begin{equation*}
X_N:=\rm{span}\big\{\phi_1,\phi_2,\dots,\phi_N\big\},\quad X_N\subset X,
\end{equation*}
where $N$ is a positive integer. Let $p_N\in X_N$ be the approximation of $p$, i.e.,
\begin{equation*}
p_N(x)=\sum_{n=1}^Nc_n^N\phi_n(x),\quad\text{ where }{\bm C^N}=(c^N_1,\dots,c^N_N)\in\R^N.
\end{equation*}
Define the following operators
\begin{equation*}
\begin{aligned}
B: \,\R^N&\rightarrow  X_N\\
{\bm C^N}&\mapsto  p_N(x)
\end{aligned}\, ;
\quad
\begin{aligned}
P: \,&\R^N\rightarrow  L^2(\Gamma_h)\\
&{\bm C^N}\mapsto  S\circ B ({\bm C^N})
\end{aligned}
\end{equation*}
Note that for simplicity, the space of the measured scattered data is changed to $L^2(\Gamma_H)$.

Then the discrete inverse problem is described the following optimatimization prblem in a finite space.

\noindent
\textbf{Discrete Inverse Problem: } To find some ${\bm C^N}\in\R^N$, such that
\begin{equation}\label{eq:dip}
\|P({\bm C^N})-U\|^2_{L^2(\Gamma_h)}=\min_{{\bm C_*^N}\in\R^N}\big\{\|P({\bm C_*^N})-U\|^2_{L^2(\Gamma_h)}\big\}.
\end{equation}

In the following, we will explain in detail the Newton-CG method to solve the discrete inverse problem \eqref{eq:dip}. The linearized equation of the discrete inverse problem \eqref{eq:dip} has the following representation
\begin{equation}\label{eq:ldip}
P({\bm C^N})+(DP)({\bm C^N})({\bm H})=U,
\end{equation}
where ${\bm H}\in\R^N$, $(DP)({\bm C^N})$ is the Fr\'{e}chet derivative of $P$ at $\bm C^N$.  $P$ is a functional defined on the finite dimensional space $\R^N$, so its $n$-th derivative 
\begin{eqnarray}
\frac{\partial P}{\partial c^N_n}=(DS)(B({\bm C^N}))\phi_n.
\end{eqnarray}
For some ${\bm H}=(h_1,h_2,\dots,h_N)$, the Fr\'{e}chet derivative
\begin{equation}\label{eq:dp}
(DP)({\bm C^N}){\bm H}=\sum_{n=1}^N h_n\frac{\partial P}{\partial c^N_n}=
(DS)(B({\bm C^N}))\left[\sum_{n=1}^N h_n\phi_n\right].
\end{equation}

For any ${\bm H}\in\R^N$ and $\phi\in L^2(\Gamma_H)$,
\begin{eqnarray*}
\left({\bm H},(DP)^*\phi\right)&=&\left((DP)({\bm C^N}){\bm H},\phi\right)\\
&=&\left(\sum_{j=1}^Nh_j\phi_j(x),M\phi\right)\\
&=&\sum_{j=1}^N h_j\left(\phi_j,M\phi\right)
\end{eqnarray*} 
Finally, 
\begin{equation}\label{eq:dpstar}
(DP)^*\phi=\Big((M\phi,\phi_1),(M\phi,\phi_2),\dots,(M\phi,\phi_N)\Big).
\end{equation}




From the representation of the operator $DP$ and its adjoint $(DP)^*$, the operators could be numerically approximated by solving the problems \eqref{eq:der1}-\eqref{eq:der3} and \eqref{eq:adj1}-\eqref{eq:adj3}. Follow the steps in Section \ref{sec:dir}, the problems could also be written as  equivalent systems of coupled quasi-periodic problems, and then the systems could be solved by the method, for more details see \cite{lechl2017}.

 With the knowledge of the operators, we can conclude the Newton-CG method to solve the discrete inverse problem.

\begin{algorithm}\caption{Newton-CG Method}
Input: Data $U$; $\epsilon>0$; $j=0$.\\
Initialization: ${\bm C^N_0}=(0,\dots,0)\in\R^N$.
\begin{algorithmic}[1] 
\While{$\|P({\bm C^N_j})-U\|_{L^2(\Gamma_H)}>\epsilon\|U\|_{L^2(\Gamma_H)}$}

\\\quad $CGNE$ iteration scheme to solve $(DP)({\bm C^N_j})({\bm H})=U-P({\bm C^N_j})$
\\\quad
${\bm C^N_{j+1}}={\bm C^N_j}+{\bm H}$;\\
\quad$j=j+1$;
\EndWhile
\end{algorithmic}
\end{algorithm}

\section{Numerical Examples}

In numerical examples, all the periodic surfaces have one fixed period $\Lambda=2\pi$. The incident field is a Herglotz wave function $v_g$ with wave number $k=3$ with
\begin{equation*}
g(x)=(x-1)^6(x+1)^6\cdot \chi_{(-1,1)}(x),
\end{equation*}
where the function $\chi_{(a,b)}$ is a function equals to $1$ in $(a,b)$, and equals to $0$ otherwise. The periodic surfaces are defined by
\begin{eqnarray*}
&&f_1(x)=2+\cos x/4;\\
&&f_2(x)=2-\cos x/4;\\
&&f_3(x)=1+\sin x/3-\cos 2x/4;
\end{eqnarray*}
while the locally perturbations are defined by
\begin{eqnarray*}
&&p_1(x)=\left(-1/4-\cos x/4\right)\cdot \chi_{(-\pi,\pi)}(x);\\
&&p_2(x)=\left(1/4+\cos x/4\right)\cdot \chi_{(-\pi,\pi)}(x);\\
&&p_3(x)=5\times 10^{-4}(x-3)^3(x-3)^3\sin\left[\frac{\pi(x+3)}{3}\right]\cdot \chi_{(-3,3)}(x).
\end{eqnarray*}
The periodic surface $\Gamma^j:=\{(x,f_j(x)):\,x\in\R\}$, while the locally perturbed surface $\Gamma^j_p:=\{(x,f_j(x)+p_j(x)):\,x\in\R\}$.

\begin{figure}[tttttt!!!b]
\centering
\begin{tabular}{c c c}
\includegraphics[width=0.3\textwidth]{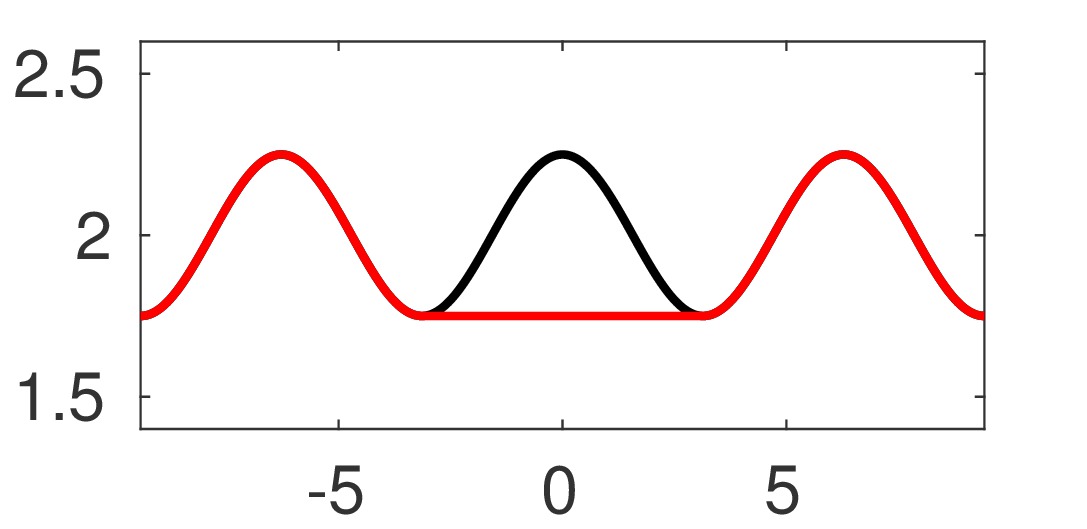} 
& \includegraphics[width=0.3\textwidth]{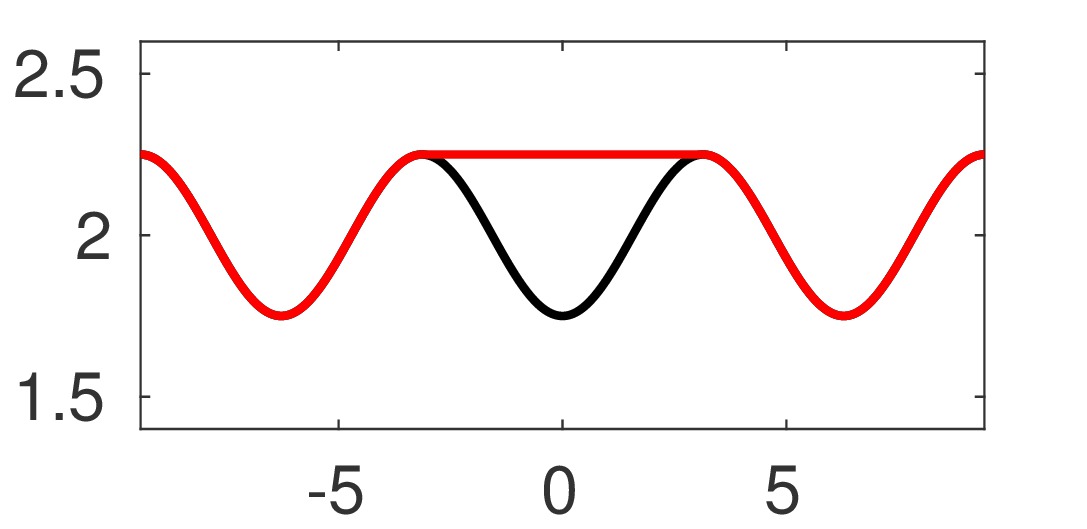} 
& \includegraphics[width=0.3\textwidth]{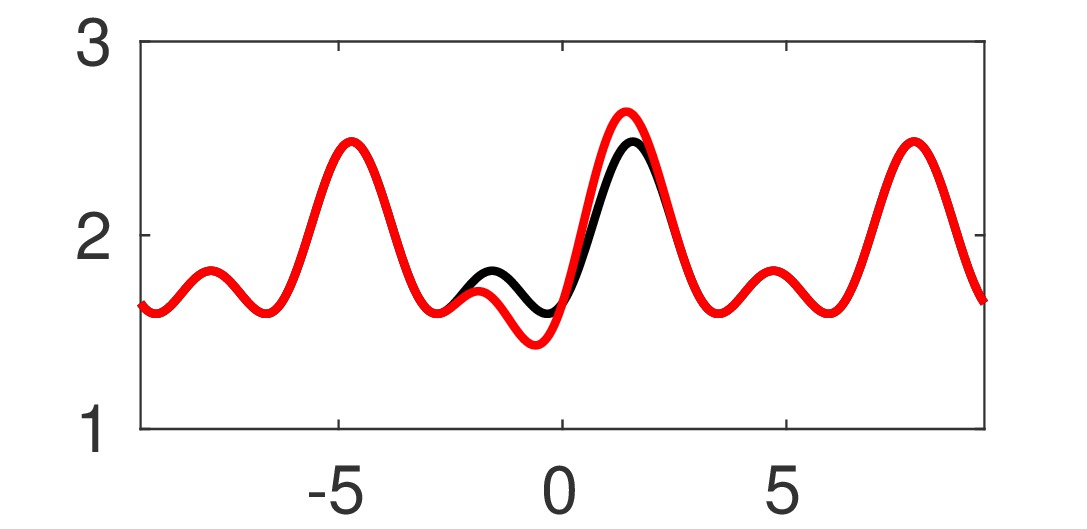}\\[-0cm]
(a) & (b) & (c) 
\end{tabular}%
\caption{(a)-(c): The three periodic surfaces $\Gamma^{1,2,3}$ (black) and $\Gamma^{1,2,3}_p$ (red) defined by the functions $f_{1,2,3}$ and $p_{1,2,3}$.}
\end{figure}

\begin{figure}[tttttt!!!b]
\centering
\begin{tabular}{c c c}
\includegraphics[width=0.3\textwidth]{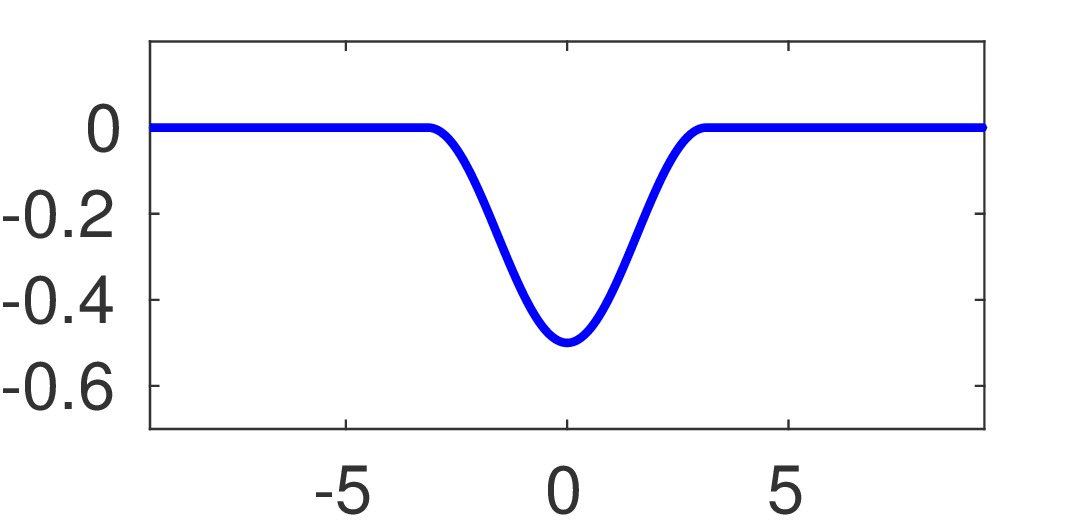} 
& \includegraphics[width=0.3\textwidth]{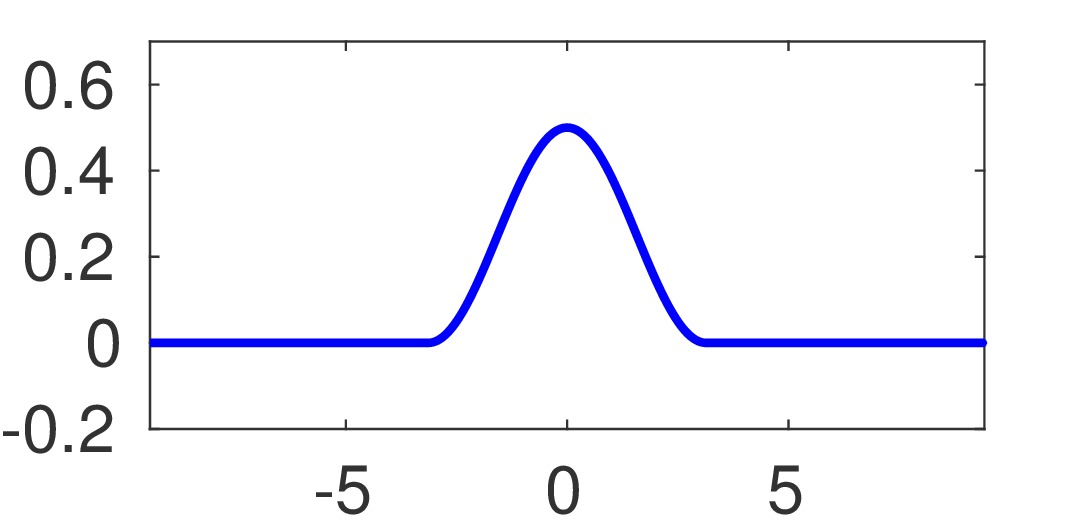} 
& \includegraphics[width=0.3\textwidth]{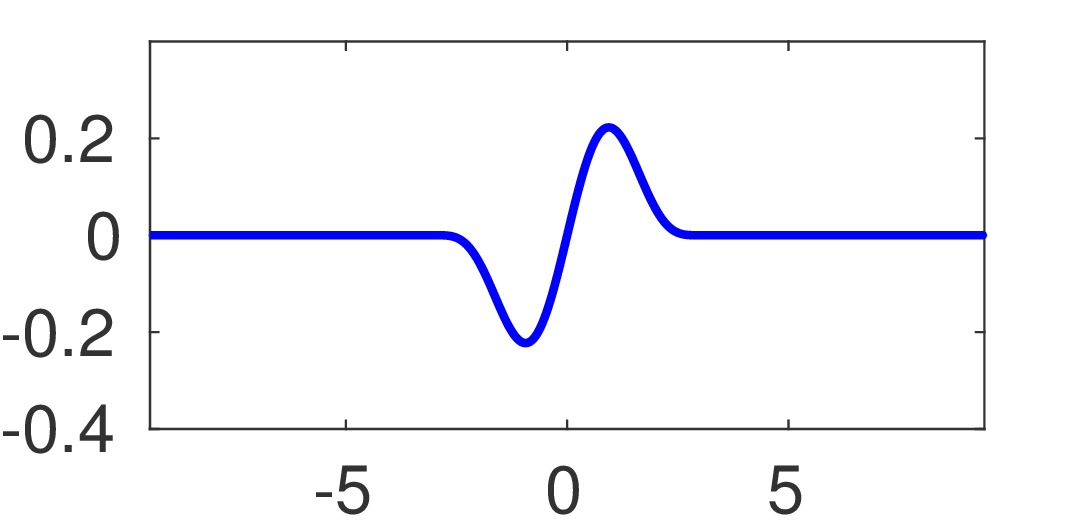}\\[-0cm]
(a) & (b) & (c) 
\end{tabular}%
\caption{(a)-(c): The three local perturbation functions $p_{1,2,3}$.}
\end{figure}

Assume that $\mathcal{M}_h$ is a quasi-uniform mesh of the periodic domain $\Omega^\Lambda_H$, for some $h>0$. We also divide the interval $\Wast$  uniformly by the grid points
\begin{equation*}
\alpha_M^{(j)}=-\frac{\pi}{\Lambda}+\frac{\pi}{M\Lambda}(2j-1)\quad j=1,2,\dots,M.
\end{equation*}
Then the problem \eqref{eq:var_blo} is discretized in this mesh, and could be solved by the finite element method introduced in  \cite{lechl2017}. 

For the numerical simulation of the measured data, we fix the meshsize $h=0.01$ and $M=80$. The scattered data is obtained by the Rayleigh expansion \eqref{eq:ray} for each quasi-periodicity $\alpha^{(j)}_M$ and the inverse Bloch transform, and is measured either on $\Gamma_H$ (near-field data) or on $\Gamma_h$ (far-field data), where $H=4$ and $h=100$. We collect the measured data in the interval $[-40\pi,40\pi]$ for the first and second structure, while in $[-160\pi,160\pi]$ for the third structure. Let $\widetilde{\Gamma}_H$ and $\widetilde{\Gamma}_h$ be the line segments. A $5\%$ random noise is also added to the scattered field, i.e.,
\begin{equation}
U^{near}=u^s_p|_{\widetilde{\Gamma}_H}(1+0.05\,c),\quad U^{far}=u^s_p|_{\widetilde{\Gamma}_h}(1+0.05\,c),
\end{equation}
where $c$ is a random function with the range $[-1,1]$.

With the measured data $U^{near}$ or $U^{far}$, we will start the reconstruction of the perturbation $p$. From the definitions of $p_1,\,p_2,\,p_3$, the integer $J=0$. For the initialization step, we show the results both from the near-field data (see Figure \ref{fig:Ini_near}) and the far-field data (see Figure \ref{fig:Ini_far}) for the three structures. From the Figures \ref{fig:Ini_near} and \ref{fig:Ini_far}, we can see that the initialization method works well for both near-field and far-field data. However, the results for simpler surfaces ($\Gamma^1_p$ and $\Gamma^2_p$) are better than results for the more complicated surface, i.e., $\Gamma^3_p$.

\begin{figure}[tttttt!!!b]\label{fig:Ini_near}
\centering
\begin{tabular}{c c c}
\includegraphics[width=0.3\textwidth]{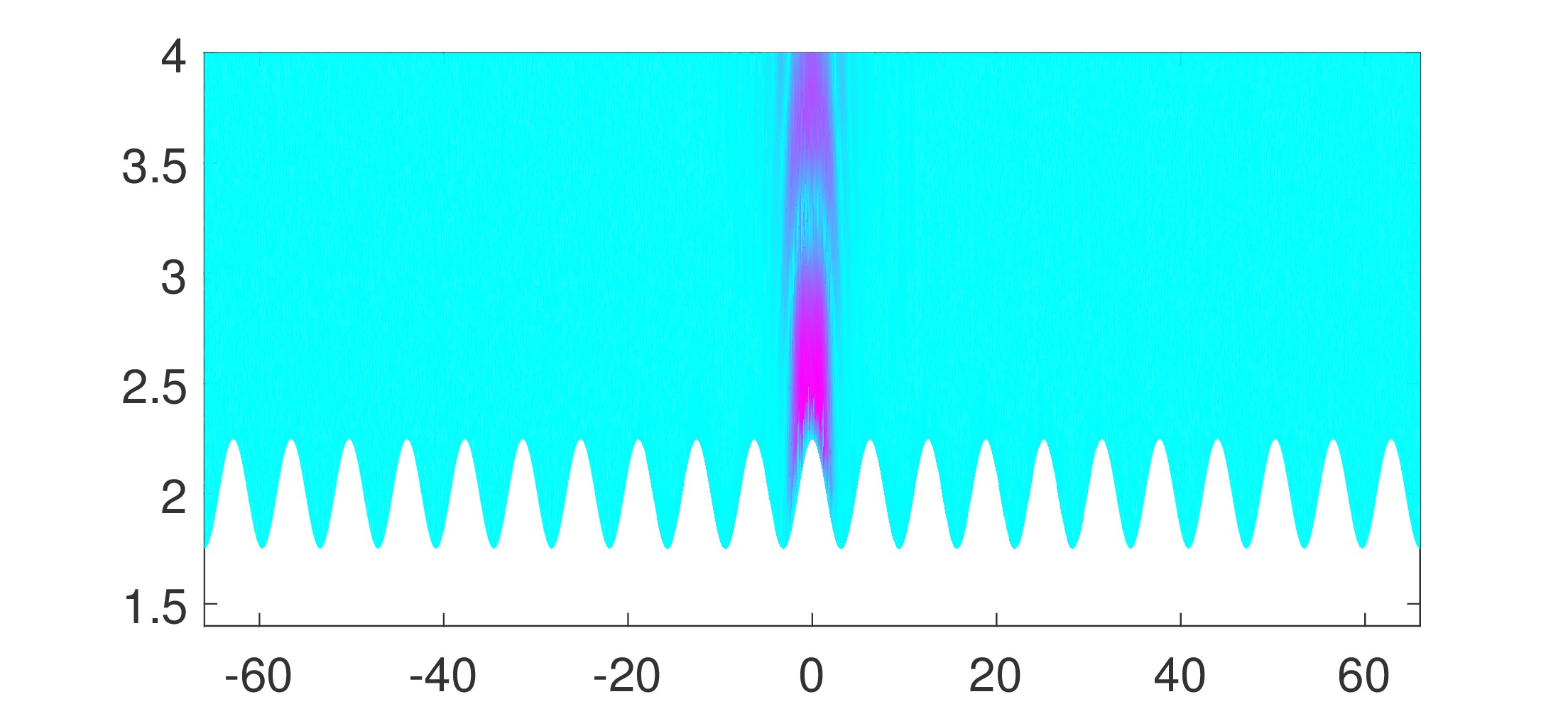} 
& \includegraphics[width=0.3\textwidth]{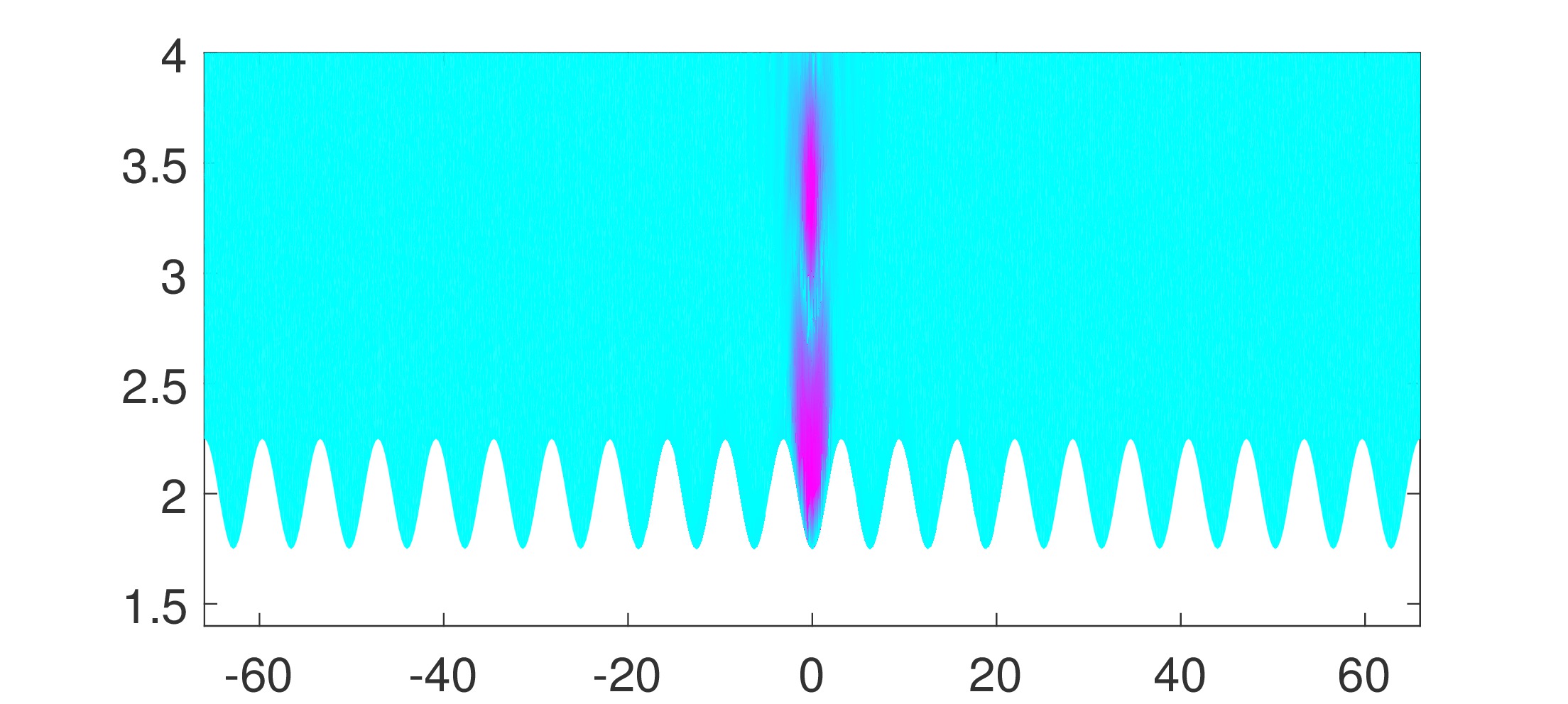} 
& \includegraphics[width=0.3\textwidth]{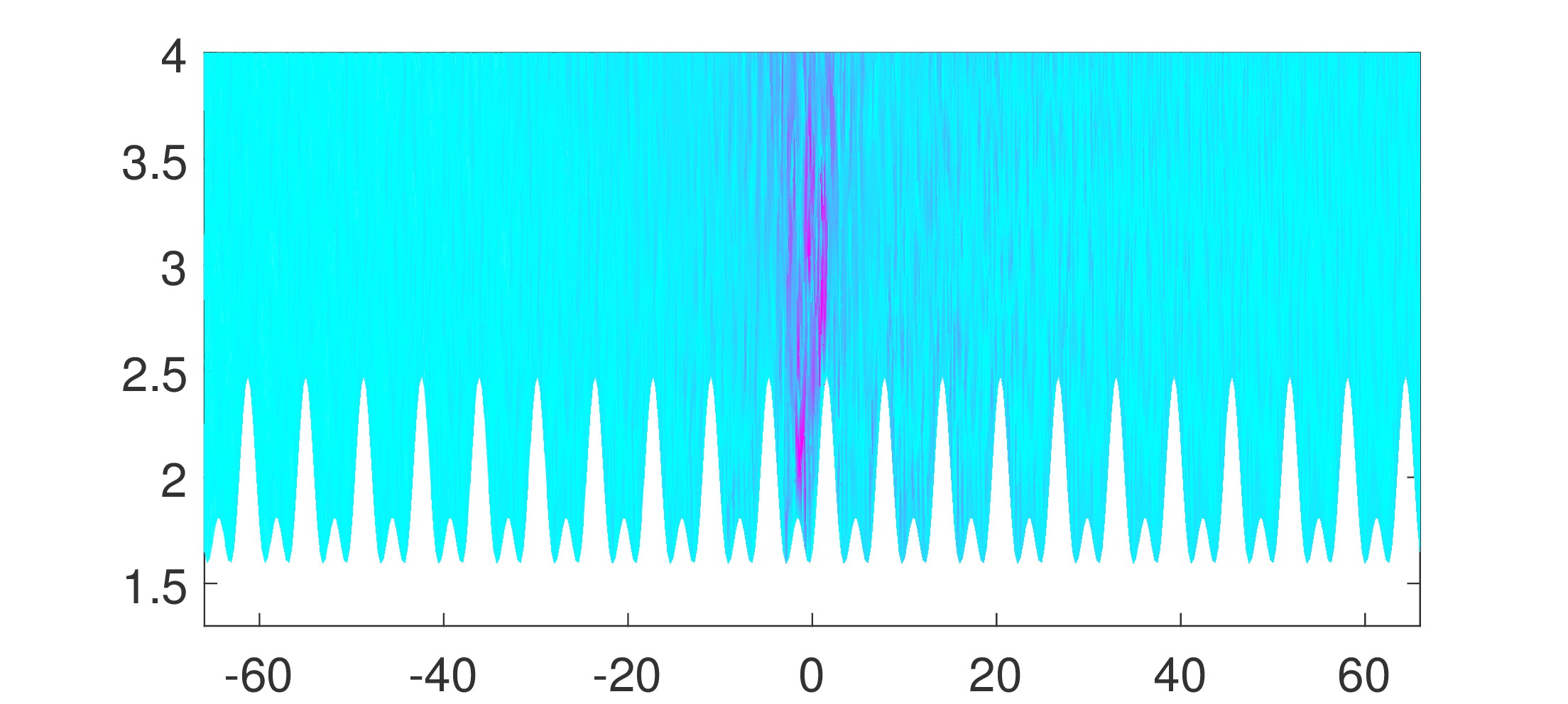}\\[-0cm]
(a) & (b) & (c) 
\end{tabular}%
\caption{(a)-(c): Initialization from near-field data for $\Gamma_p^{1,2,3}$.}
\end{figure}

\begin{figure}[tttttt!!!b]\label{fig:Ini_far}
\centering
\begin{tabular}{c c c}
\includegraphics[width=0.3\textwidth]{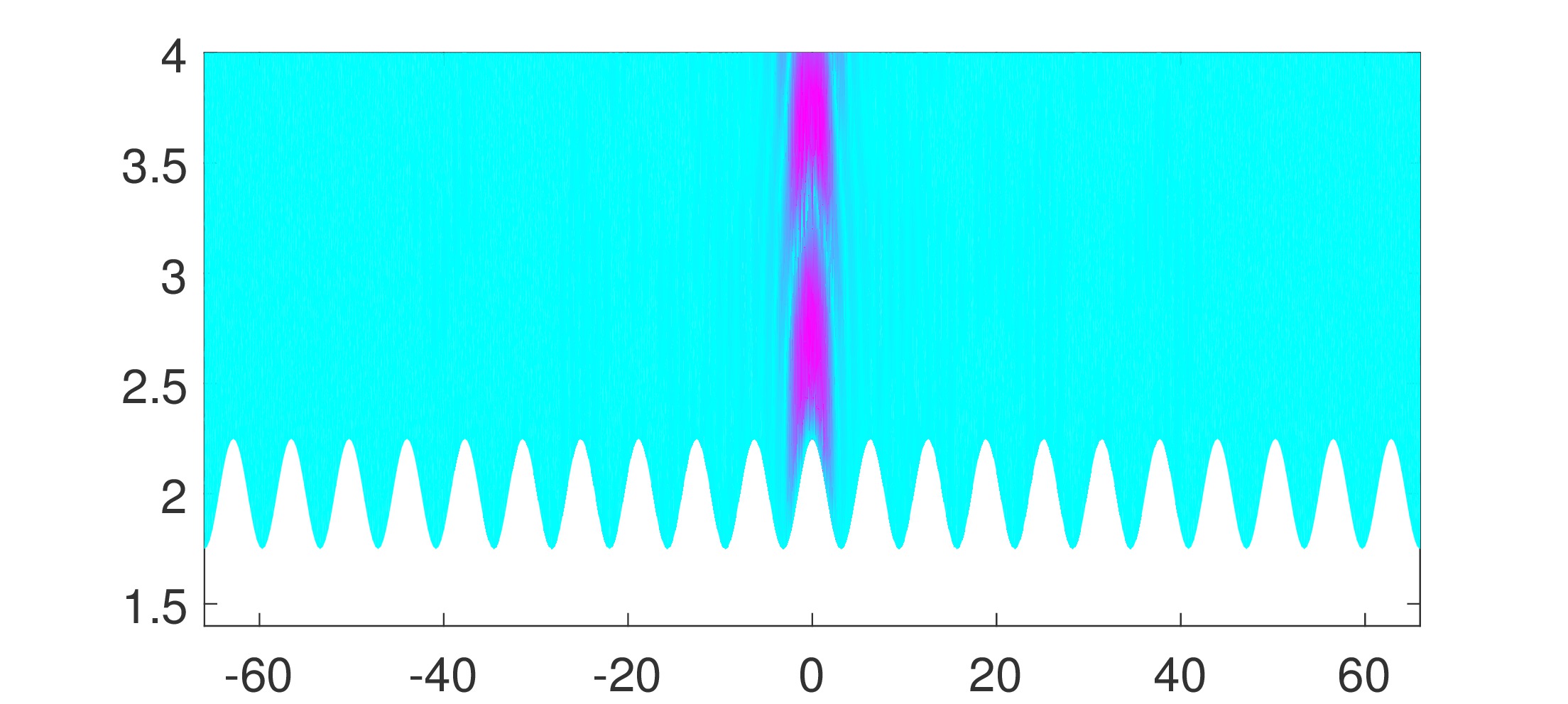} 
& \includegraphics[width=0.3\textwidth]{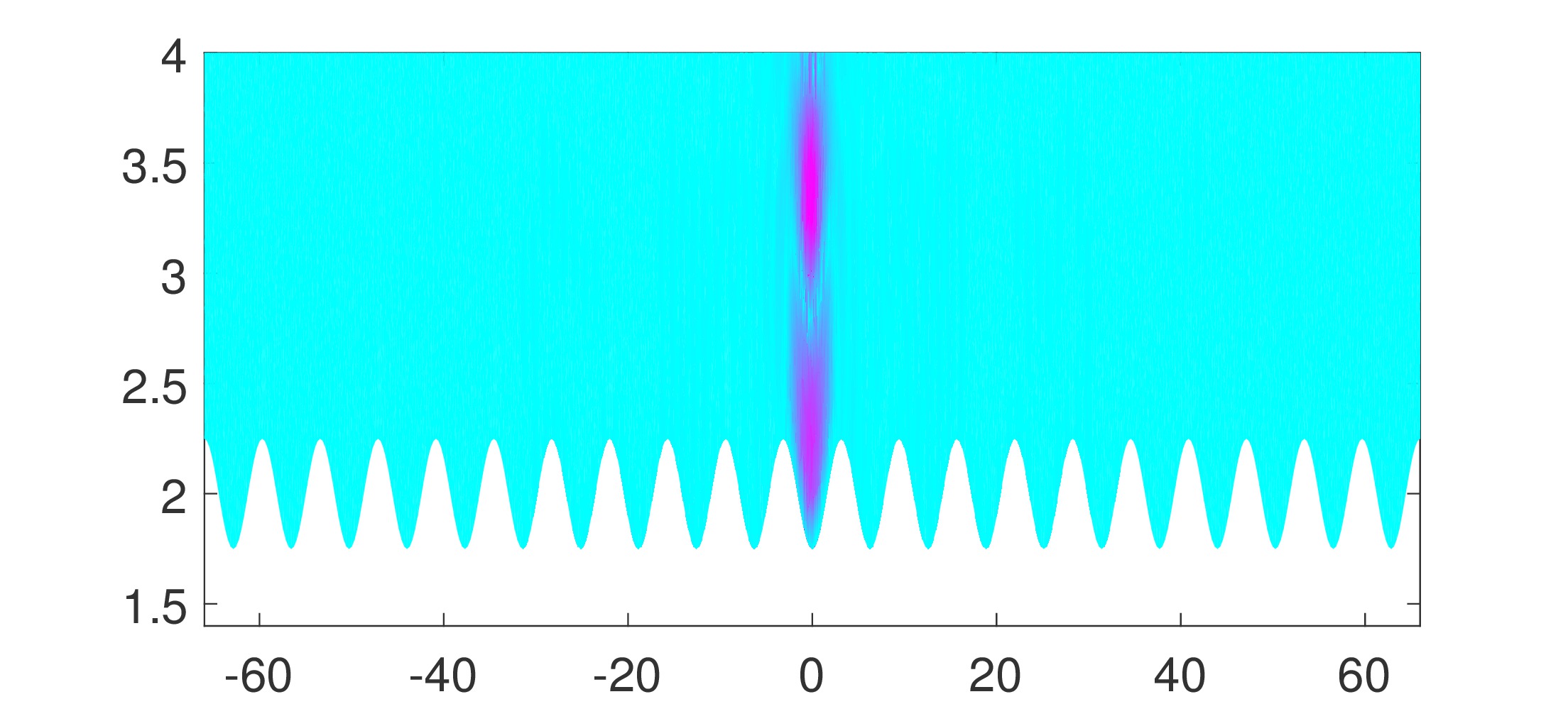} 
& \includegraphics[width=0.3\textwidth]{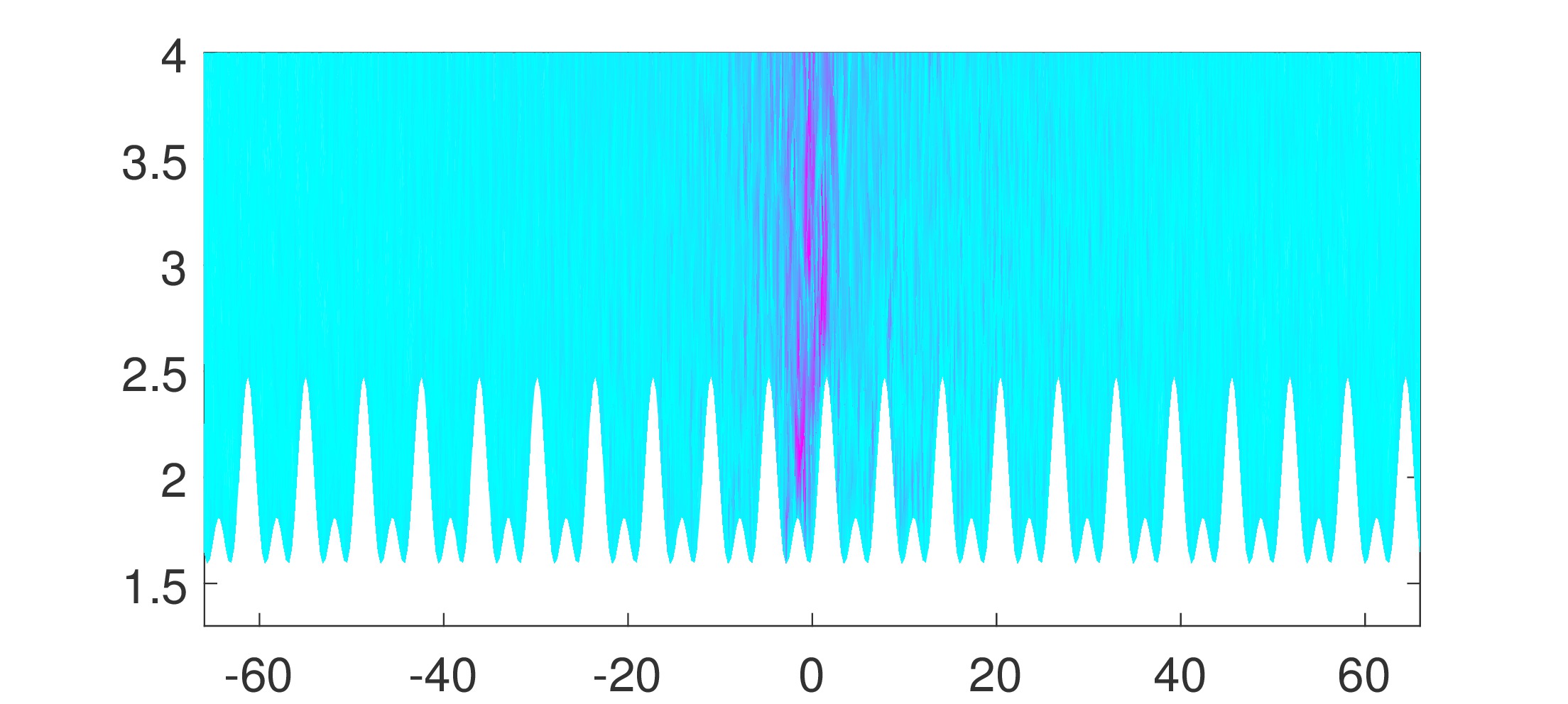}\\[-0cm]
(a) & (b) & (c) 
\end{tabular}%
\caption{(a)-(c): Initialization from far-field data for $\Gamma_p^{1,2,3}$.}
\end{figure}

In the Newton-CG algorithm to reconstruct the function $p$, spline basis functions are chosen as basic functions to approximate $p$. The function $p$ is approximated by the finite series with $N=10$ basic functions. 

During the Newton-CG scheme, the problems \eqref{eq:der1}-\eqref{eq:der3} and \eqref{eq:adj1}-\eqref{eq:adj3} could be reformulated as a coupled family of quasi-periodic problems, with the similar techniques in Section \ref{sec:dir}. To solve these problems, the mesh size is chosen as $h=0.02$ with $N=80$. 
\begin{remark}
For the measured data $U^{far}$, to solve the problem in a relatively smaller domain, we will define the approximated near field data $\widetilde{U}^{near}$ on $\Gamma_H$, and then solve the problems in $\Omega^\Lambda_H$ by using the field $\widetilde{U}^{near}$. As the Bloch transform $\left(\J_\Omega U^{far}\right)(\alpha,\cdot)$ has the form of
\begin{equation*}
\left(\J_\Omega U^{far}\right)(\alpha,\cdot)=\sum_{j\in\Z}\hat{W}_j e^{\i(\Lambda^*j-\alpha)x_1}.
\end{equation*}
For each $\alpha\in\Wast$, $\left(\J_\Omega U^{far}\right)(\alpha,\cdot)$ has evanescent modes that decreases very fast,  we define the approximated fields by removing such modes, i.e.,
\begin{equation*}
W(\alpha,x):=\sum_{j\in\Z:\,|\Lambda^*j-\alpha|\leq k}\hat{W}_j e^{\i(\Lambda^*j-\alpha)x_1+\i\sqrt{k^2-|\Lambda^*j-\alpha|^2}(x_2-h)}.
\end{equation*}
Thus the approximated near field by
\begin{equation*}
\widetilde{U}^{near}=\left(\J_\Omega^{-1} W\right)(x_1,H).
\end{equation*}
\end{remark}

The numerical results from Newton-CG method are shown in Figure \ref{fig:New_CG_near} with near-field data, while in Figure \ref{fig:New_CG_far} with far-field data. For the simpler surfaces, i.e., for the first and second structure, the reconstruction is very well, for both near-field data and far-field data. However, for the third structure, it is easier for us to obtain a better reconstruction from the near-field data. This is reasonable for the far-field data may lost some information from the surface.

\begin{figure}[tttttt!!!b]\label{fig:New_CG_near}
\centering
\begin{tabular}{c c c}
\includegraphics[width=0.3\textwidth]{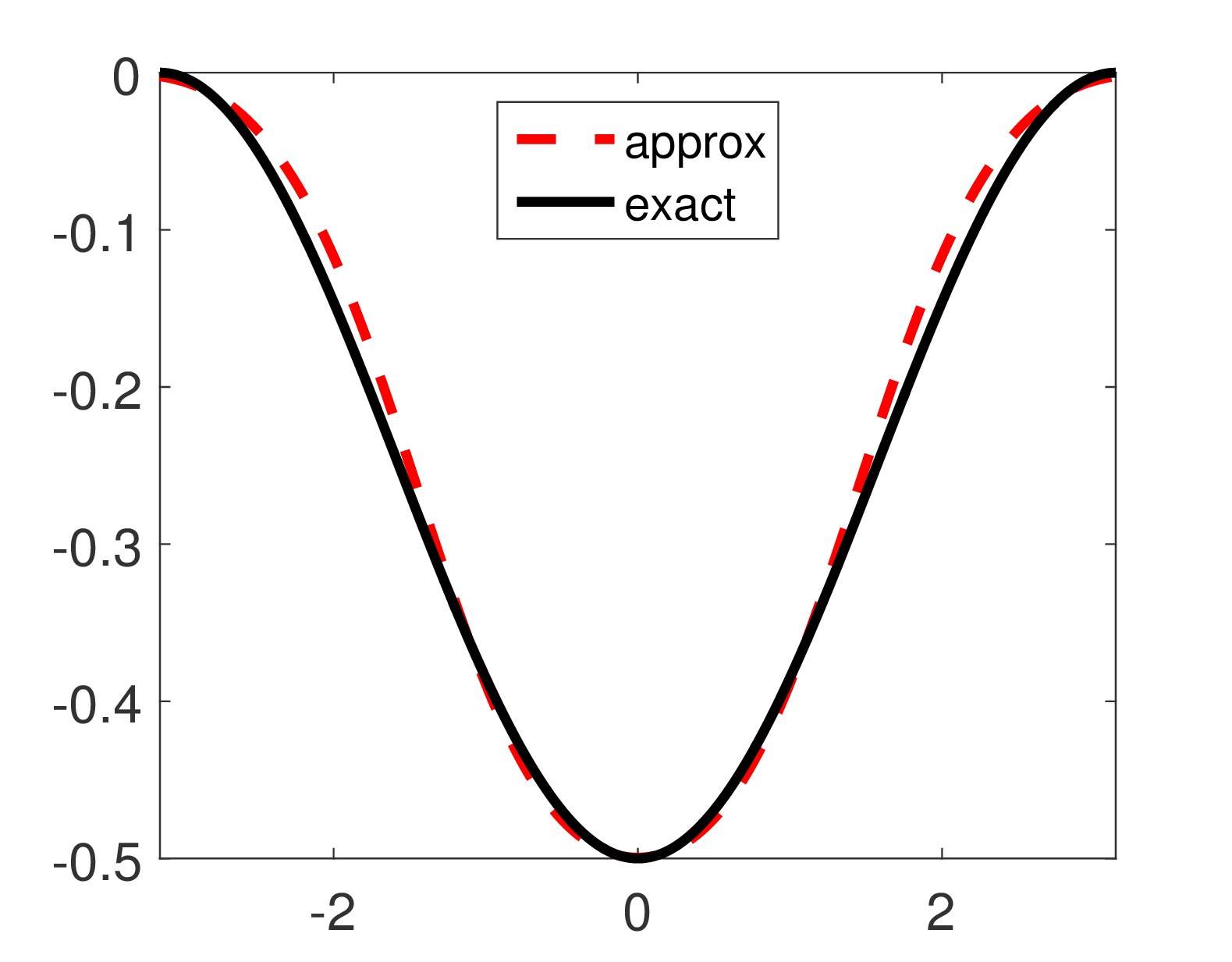} 
& \includegraphics[width=0.3\textwidth]{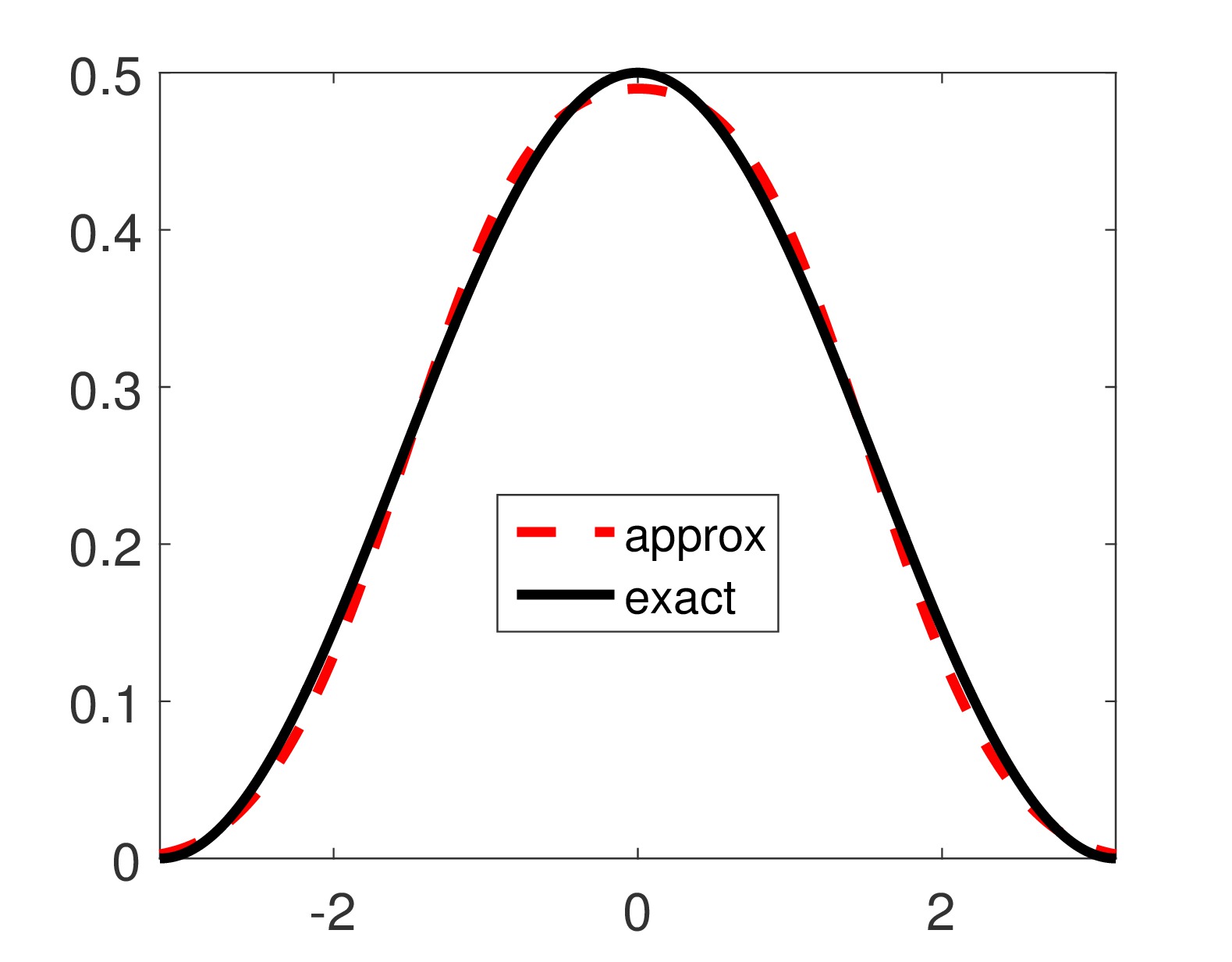} 
& \includegraphics[width=0.3\textwidth]{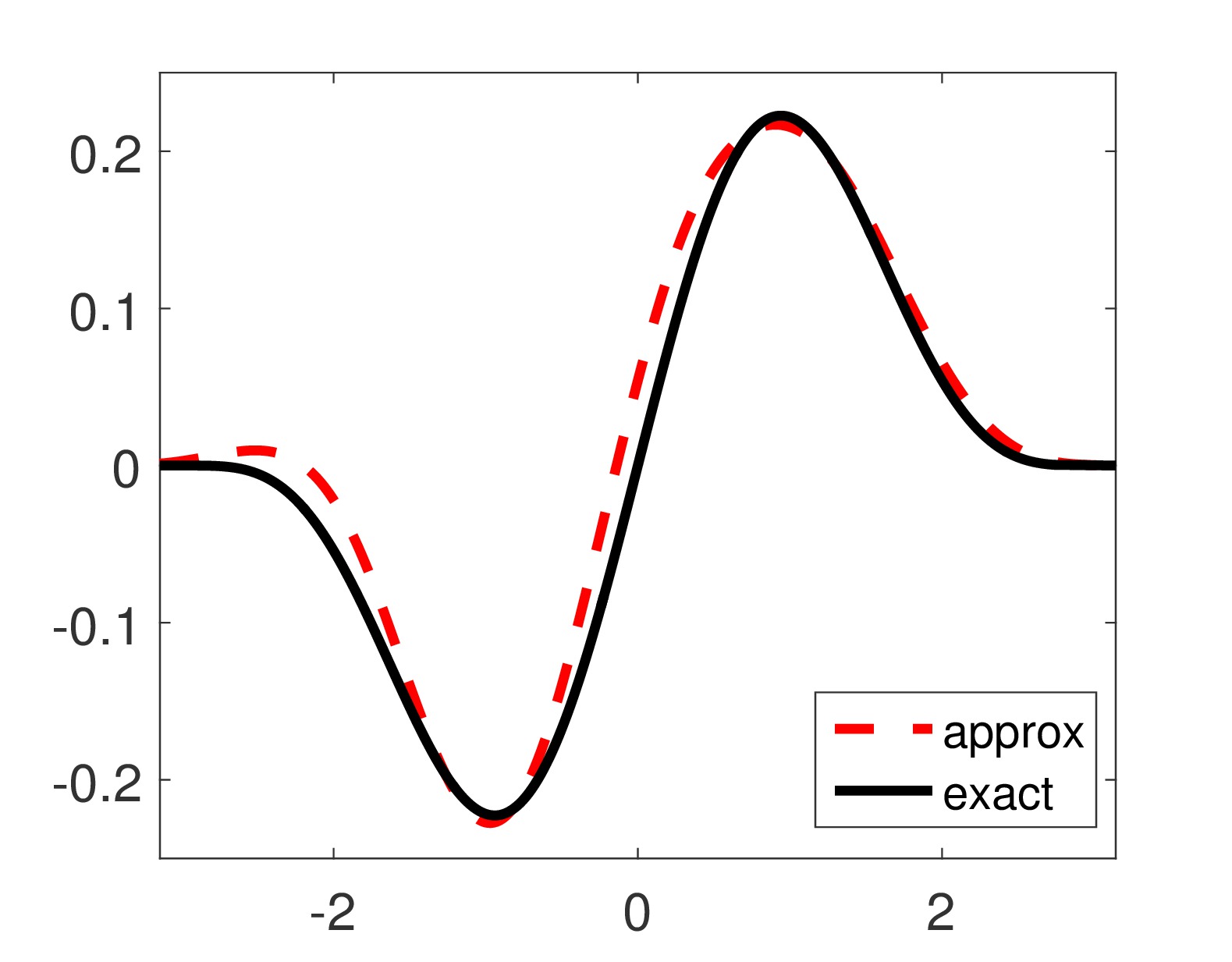}\\[-0cm]
(a) & (b) & (c) 
\end{tabular}%
\caption{(a)-(c): The reconstructions of $p_{1,2,3}$ from near-field data.}
\end{figure}

\begin{figure}[tttttt!!!b]\label{fig:New_CG_far}
\centering
\begin{tabular}{c c c}
\includegraphics[width=0.3\textwidth]{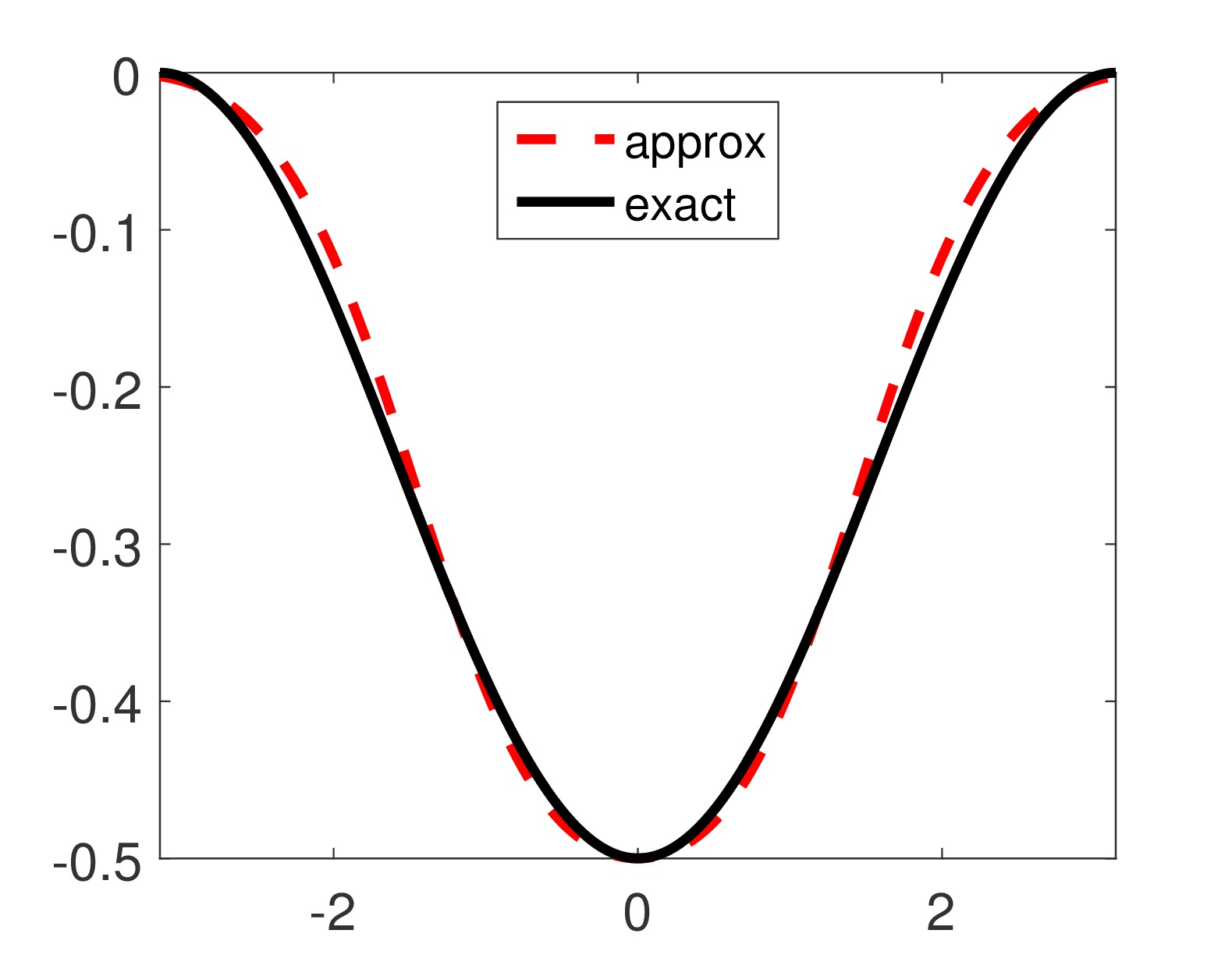} 
& \includegraphics[width=0.3\textwidth]{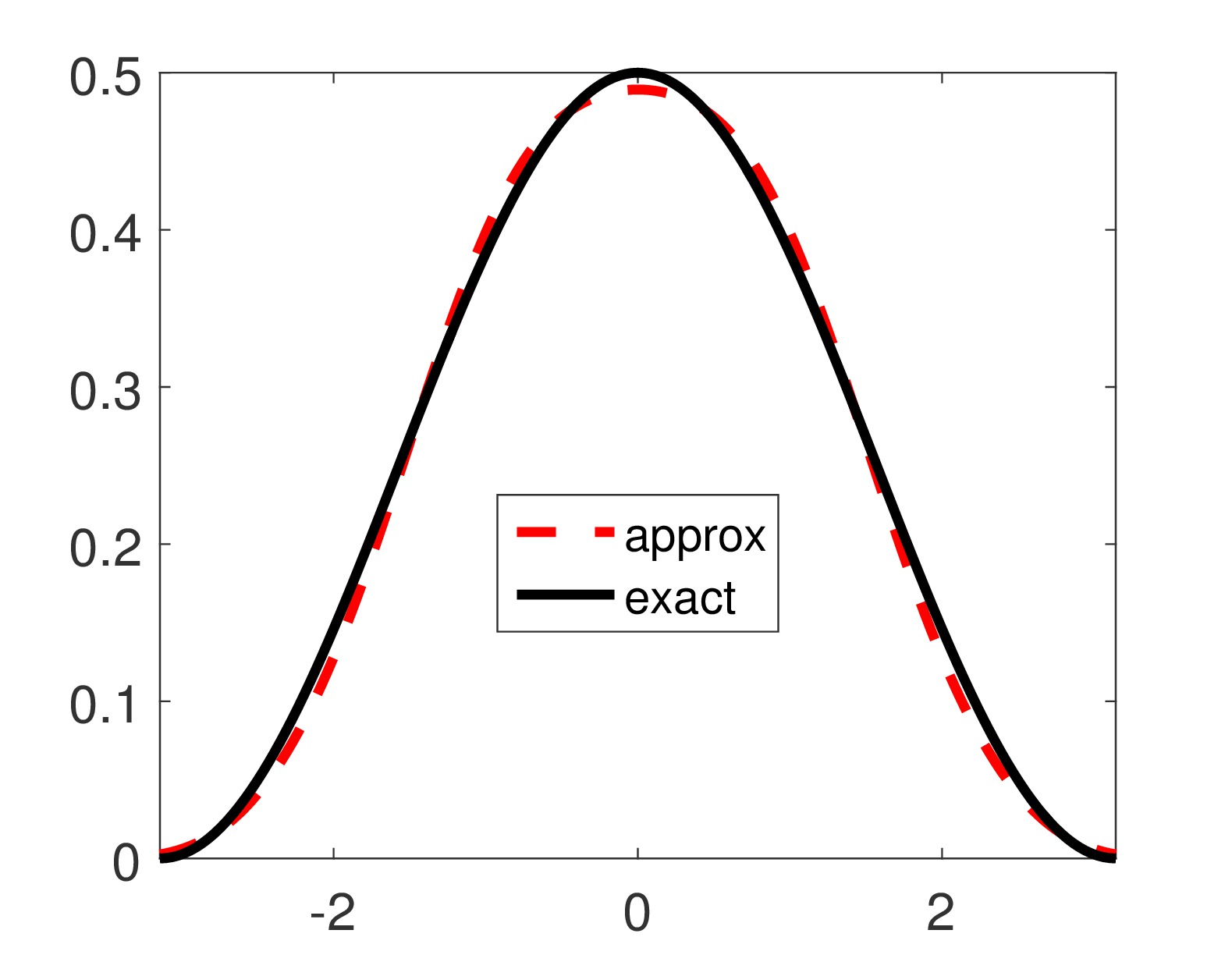} 
& \includegraphics[width=0.3\textwidth]{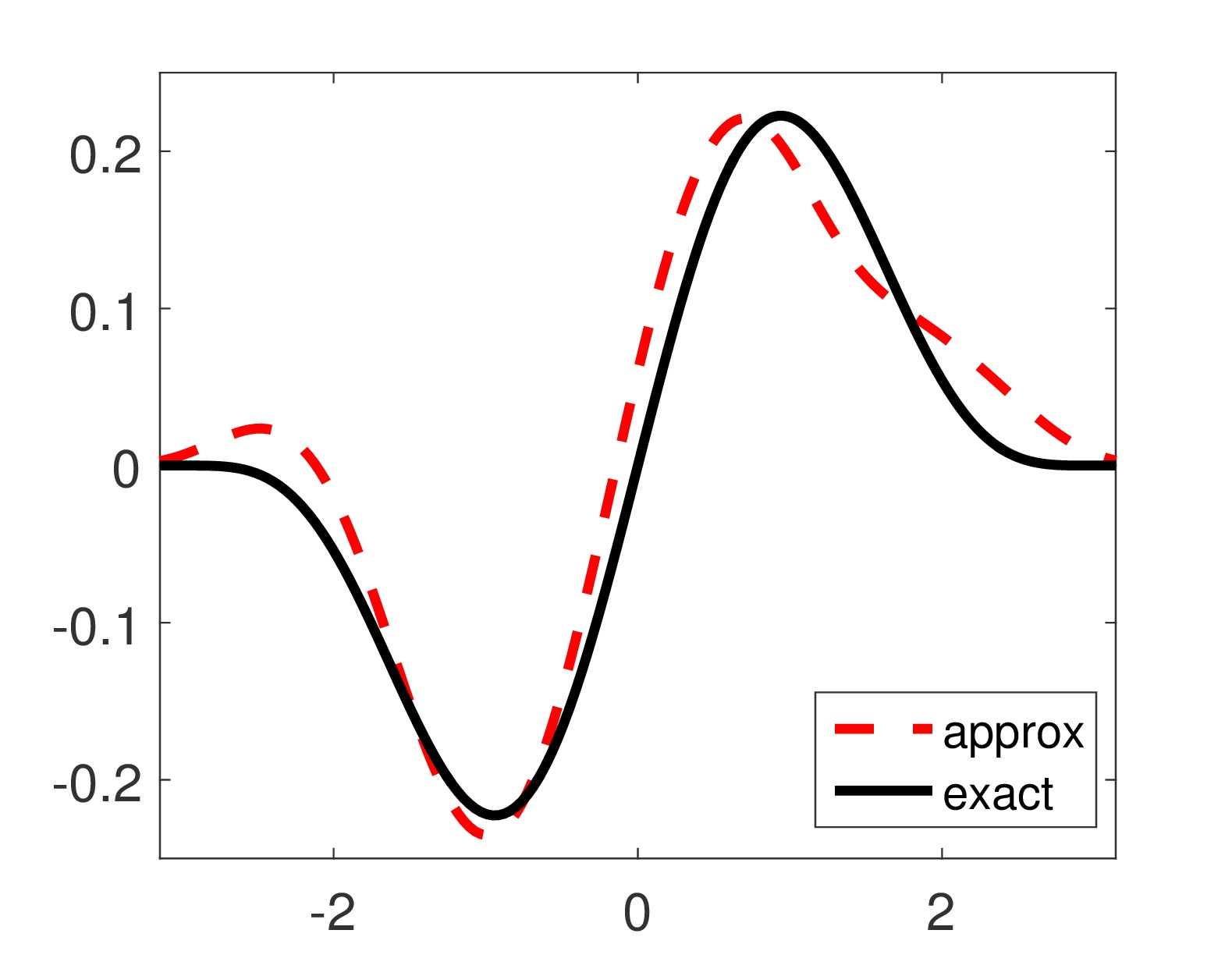}\\[-0cm]
(a) & (b) & (c) 
\end{tabular}%
\caption{(a)-(c): The reconstructions of $p_{1,2,3}$ from near-field data.}
\end{figure}

\section*{Acknowlegdments}
The second author was supported by the University of Bremen and the European Union FP7 COFUND under grant agreement n$^o$ 600411.

\bibliographystyle{alpha}
\bibliography{ip-biblio} 

\providecommand{\noopsort}[1]{}
\begin{thebibliography}{EHN96}

\bibitem[CE10]{Chand2010}
S.~N. {Chandler-Wilde} and J.~Elschner.
\newblock Variational approach in weighted {S}obolev spaces to scattering by
  unbounded rough surfaces.
\newblock {\em SIAM. J. Math. Anal.}, 42:2554--2580, 2010.

\bibitem[Coa12]{Coatl2012}
J.~Coatl{\'e}ven.
\newblock {Helmholtz equation in periodic media with a line defect}.
\newblock {\em {J. Comp. Phys.}}, 231:1675--1704, 2012.

\bibitem[EHN96]{Engl1996}
Heinz~W. Engl, Martin Hanke, and Andreas Neubauer.
\newblock {\em Regularization of inverse problems}.
\newblock Kluwer Acad. Publ., Dordrecht, Netherlands, 1996.

\bibitem[HN16]{Hadda2016}
H.~Haddar and T.~P. Nguyen.
\newblock {A volume integral method for solving scattering problems from
  locally perturbed infinite periodic layers}.
\newblock {\em Appl. Anal.}, 96(1):130--158, 2016.

\bibitem[IJZ12a]{Ito2012}
K.~Ito, B.~Jin, and J.~Zou.
\newblock A direct sampling method to an inverse medium scattering problem.
\newblock {\em Inverse Problems}, 28:025003(11pp), 2012.

\bibitem[IJZ12b]{Ito2012a}
K.~Ito, B.~Jin, and J.~Zou.
\newblock A two-stage method for inverse medium scattering.
\newblock {\em J. Comput. Phys.}, 237:211--223, 2012.

\bibitem[IJZ13]{Ito2013}
K.~Ito, B.~Jin, and J.~Zou.
\newblock A direct sampling method for inverse electromagnetic medium
  scattering.
\newblock {\em Inverse Problems}, 29:095018(19pp), 2013.

\bibitem[Kir93]{Kirsc1993b}
A.~Kirsch.
\newblock The domain derivative and two applications in inverse scattering
  theory.
\newblock {\em Inverse Problems}, 9:81--96, 1993.

\bibitem[Lec08]{Lechl2008a}
A.~Lechleiter.
\newblock {\em Factorization Methods for Photonics and Rough Surface
  Scattering}.
\newblock PhD thesis, Universit{\"a}t Karlsruhe, Karlsruhe, Germany, 2008.

\bibitem[Lec17]{lechl2016}
A.~Lechleiter.
\newblock The {F}loquet-{B}loch transform and scattering from locally perturbed
  periodic surfaces.
\newblock {\em J. Math. Anal. Appl.}, 446(1):605--627, 2017.

\bibitem[LN15]{Lechl2015e}
A.~Lechleiter and D.-L. Nguyen.
\newblock {Scattering of {H}erglotz waves from periodic structures and mapping
  properties of the {B}loch transform}.
\newblock {\em {Proc. Roy. Soc. Edinburgh Sect. A}}, 231:1283--1311, 2015.

\bibitem[LZ13]{Li2013}
J.~Li and J.~Zou.
\newblock A direct sampling method for inverse scattering using far-field data.
\newblock {\em Inverse Problems and Imaging}, 7:095018(19pp), 2013.

\bibitem[LZ16]{Lechl2016b}
A.~Lechleiter and R.~Zhang.
\newblock Non-periodic acoustic and electromagnetic scattering from periodic
  structures in 3d.
\newblock {\em Accepted by Comput. Math. Appl.}, 2016.

\bibitem[LZ17a]{Lechl2016a}
A.~Lechleiter and R.~Zhang.
\newblock A convergent numerical scheme for scattering of aperiodic waves from
  periodic surfaces based on the {F}loquet-{B}loch transform.
\newblock {\em SIAM J. Numer. Anal}, 55(2):713--736, 2017.

\bibitem[LZ17b]{lechl2017}
A.~Lechleiter and R.~Zhang.
\newblock A {F}loquet-{B}loch transform based numerical method for scattering
  from locally perturbed periodic surfaces.
\newblock {\em Accepted for SIAM J. Sci. Comput.},
  https://arxiv.org/abs/1611.06360, 2017.

\end{thebibliography}

\end{document}